\documentclass[a4paper,12pt,leqno, twoside]{article}
\usepackage{amssymb}
\usepackage{amsthm}
\usepackage[mathscr]{eucal}

\def\NM{{\mathbb{N}}}

\def\PM{{\mathbb{P}}}
\def\RM{{\mathbb{R}}}

\def\QM{{\mathbb{Q}}}
\def\FM{{\mathbb{F}}}
\def\ZM{{\mathbb{Z}}}

\def\mG{{\mathfrak m}}

\def\SG{{\mathfrak S}}

\def\CC{{\mathcal C}}
\def\HC{{\mathcal H}}
\def\NC{{\mathcal N}}

\def\RC{{\mathcal R}}

\def\OC{{\mathcal O}}
\def\MC{{\mathcal M}}
\def\KC{{\mathcal K}}
\def\PC{{\mathcal P}}

\def\LC{{\mathcal L}}
\def\EC{{\mathcal E}}

\def\subsetneq{\varsubsetneq}

\def\simto{\buildrel\hbox{\tiny{$\sim$}}\over\longrightarrow}

\def\leq{\leqslant}
\def\geq{\geqslant}
\def\injo{\hookrightarrow}
\def\id{\mathop{\mathrm{Id}}\nolimits}

\def\ba{\backslash}
\def\wt{\widetilde}
\def\wh{\widehat}

\def\application#1#2#3#4#5{\begin{array}{rcl}
                            #1 \;\;\; #2 & \to &  #3 \\
                              #4 & \mapsto & #5 
                            \end{array}} 
\def\cas#1#2#3#4#5{\begin{array}{rcl} #1 \; & = &
    \left\{\begin{array}{rcl} #2 & \hbox{ if } & #3 \\
                             #4 & \hbox{ if } & #5  \end{array}
                         \right. \end{array}}

\def\To#1{\buildrel\hbox{\tiny{$#1$}}\over\longrightarrow}
\def\to{\rightarrow}

\def\hom#1#2#3{\hbox{\rm Hom}_{#3}\>\!\left(#1,#2\right)}
\def\endo#1#2{\hbox{\sl End}_{#1}\>\!\left(#2\right)}

\def\Hom{\mathop{\hbox{\rm Hom}}\nolimits}

\def\Rhom{\mathop{R\hbox{\rm Hom}}\nolimits}

\def\Rep{{\rm {R}ep}}
\def\Mo#1#2{\mathop{\Rep^{\infty}_{#1}(#2)}}

\def\Irr#1#2{\mathop{{\rm Irr}_{#1}\left(#2\right)}}

\def\ind#1#2#3{\hbox {\rm Ind}_{#1}^{#2}\>\!\left(#3\right)}  
\def\cInd#1#2{\hbox {\rm ind}_{#1}^{#2}}
\def\cind#1#2#3{\hbox {\rm ind}_{#1}^{#2}\>\!\left(#3\right)} 
\def\ip#1#2#3{\hbox {\sl i}_{#1}^{#2}\>\!(#3)}  

\def\dim{\mathop{\mbox{\rm dim}}\nolimits}

\def\val{\mathop{{\rm val}}\nolimits}

\def\ext#1#2#3#4{\mbox{\rm Ext}^{#1}_{#4}\>\!\left(#2,#3\right)}

\def\exp{\mathop{\EC\hbox{\rm xp}\,}}

\makeatletter
\renewcommand{\subsubsection}{\@startsection{subsubsection}{3}{\parindent}{-\baselineskip}{-0.01\baselineskip}{\bf}}
\renewcommand*{\@seccntformat}[1]{%
  \csname the#1\endcsname\
}
\makeatother  

\def\ali{\subsubsection{}\setcounter{equation}{0}}
\def\alin#1{\setcounter{equation}{0}\subsubsection{\it  #1}. }

\newtheoremstyle{th}
  {\baselineskip}{.5\baselineskip}{\itshape}
  {\parindent}{\bf}
  {--}{.5em}{}

\newtheoremstyle{def}
  {\baselineskip}{\baselineskip}{}
  {\parindent}{\bf}
  {--}{.5em}{}

\newtheoremstyle{th*}
  {.5\baselineskip}{.5\baselineskip}{\itshape}
  {\parindent}{\bf}
  {--}{.5em}{}

\newtheoremstyle{remark*}
  {.5\baselineskip}{.5\baselineskip}{}
  {\parindent}{\bf}
  {--}{.5em}{}

\newtheoremstyle{remark}
  {.5\baselineskip}{.5\baselineskip}{}
  {\parindent}{\bf}
  {--}{.5em}{}

\swapnumbers
\theoremstyle{th}
\newtheorem{theo}[subsubsection]{\sc Theorem.\bf}

\newtheorem{prop}[subsubsection]{\sc Proposition.\bf}

\theoremstyle{def}

\theoremstyle{remark}

\theoremstyle{th*}
\newtheorem*{thm}{\sc Theorem.}
\newtheorem*{lem}{\sc Lemma.}
\newtheorem*{pro}{\sc Proposition.}
\newtheorem*{cor}{\sc Corollary.}

\newtheorem*{defn}{\sc Definition.}
\newtheorem*{fac}{\sc Fact.}

\theoremstyle{remark*}
\newtheorem*{rak}{\sc Remark.}

\newtheorem*{no}{\sc Notation.}

\newcommand{\findem}{\hfill$\Box$\par\medskip}
\newcommand{\dem}{\indent {\it Preuve :} \rm }

\usepackage[arrow,matrix]{xy} 

\title{A geometric interpretation of the nilpotent part of local Langlands correspondence modulo
  $\ell$}
\author{Jean-Fran\c cois Dat}
\date{}

\begin{document}
\maketitle
\bibliographystyle{plain}

\def\la{\langle}
\def\ra{\rangle}
\def\knr{{\wh{K^{nr}}}}
\def\ka{\wh{K^{ca}}}

\abstract{ Let $p$ and $\ell$ be two distinct primes.
The aim of this paper is to show 
how, under a certain congruence hypothesis,  the mod $\ell$
  cohomology complex of the Lubin-Tate tower, together with a natural
  Lefschetz operator, provides a
  geometric interpretation of Vign\'eras' local Langlands correspondence
  modulo $\ell$.
}

\def\dd{D_d^\times}
\def\mdro{\MC_{Dr,0}}
\def\mdrn{\MC_{Dr,n}}
\def\mdr{\MC_{Dr}}
\def\mlto{\MC_{LT,0}}
\def\mltn{\MC_{{\rm LT},n}}
\def\mltno{\MC_{{\rm LT},n}^{(0)}}
\def\mlt{\MC_{\rm LT}}
\def\mltK{\MC_{LT,K}}
\def\LJ{{\rm LJ}}
\def\JL{{\rm JL}}
\def\SL{{\rm SL}}
\def\GL{{\rm GL}}

\def\Ql{\QM_{\ell}}
\def\Zl{\ZM_{\ell}}
\def\Fl{\FM_{\ell}}
\def\oQl{\overline\QM_{\ell}}
\def\oZl{\overline\ZM_{\ell}}
\def\oFl{\overline\FM_{\ell}}
\def\mltnc{\wh\MC_{{\tiny{\rm LT}},n}}
\def\mltnco{\wh\MC_{{\tiny{\rm LT}},n}^{(0)}}

\def\sp{{\rm sp}}

\tableofcontents

\section{Main theorem}
Let $K$ be a local non-archimedean field with
ring of integers $\OC$ and residue field $k\simeq \FM_{q}$, 
$q$ a power of a prime $p$. Let $\ell$ be another prime number and $d$
an integer. As in \cite{ltmodl}, we consider the cohomology complex
$$ R\Gamma_{c}(\mlt^{\rm ca},\Zl) \in D^{b}(\Rep^{\infty,
  c}_{\Zl}(G\times D^{\times}\times W_{K}))$$
of the height $d$ Lubin-Tate tower of $K$. Here $G=\GL_{d}(K)$, $D$ is
the division algebra which is central over $K$ with invariant $1/d$,
and $W_{K}$ is the Weil group of $K$. The category $\Rep^{\infty,c}_{\Zl}$
consists of $\Zl$-representations of the triple product which are smooth for
$G$ and $D^{\times}$ and continuous for $W_{K}$.
In \cite{lefschetz}, we defined a Lefschetz operator 
$$ L:\, R\Gamma_{c}(\mlt^{\rm ca},\Zl) \To{}  R\Gamma_{c}(\mlt^{\rm
  ca},\Zl)[2](1) $$
as the cup-product by the Chern class of the tautological
 invertible sheaf on the associated Gross-Hopkins period
domain. 

To an irreducible $\oFl$-representation $\pi$ of $G$, we
associate its ``derived $\pi$-coisotypical part''
$$ R_{\pi}:=\Rhom_{\Zl G}(R\Gamma_{c}(\mlt^{\rm ca},\Zl), \pi) \in
D^{b}(\Rep^{\infty}_{\oFl}(D^{\times}\times W_{K}))$$ 
which inherits a morphism $L_{\pi}:\, R_{\pi}\To{}R_{\pi}[2](1)$. 
We also denote by $R^{*}_{\pi}$ the total
cohomology of  $R_{\pi}$, a smooth graded $\oFl$-representation of
$W_{K}\times D^{\times}$, and by $L^{*}_{\pi}:\,
R^{*}_{\pi}\To{}R^{*}_{\pi}[2](1)$ the corresponding morphism.
Our aim here is to prove the following theorem, where we forget the grading.

\begin{thm}
  Assume that the multiplicative order of $q$ mod $\ell$ is $d$. Then
  for any unipotent irreducible representation $\pi$ of $\GL_{d}(K)$,
  there is an isomorphism
$$ (R_{\pi}^{*},{L}_{\pi}^{*})^{\rm ss}\simeq 
|\LJ(\pi)|\otimes(\sigma^{\rm ss}(\pi),L(\pi)).$$
\end{thm}

The congruence condition on $q$ modulo $\ell$ will be called the
\emph{Coxeter congruence relation}, by analogy with the modular
Deligne-Lusztig theory where this condition arises in the context of
 Brou\'e's conjecture, see \cite{Dudas1} for example.
The terminology \emph{unipotent} was introduced by Vign\'eras to denote
representations that belong to the same block as the trivial representation.
Finite group theorists would rather call them \emph{principal block} representations.

\medskip

Let us explain the notation of the theorem. The symbol $\LJ(\pi)$ stands
for the Langlands-Jacquet transfer of \cite{jlmodl}. In general it is a virtual
$\oFl$-representation of $D^{\times}$, but under the congruence
hypothesis, it is known to be effective up to sign, \emph{cf}
\cite[(3.2.5)]{jlmodl}, so we
can put $\LJ(\pi)=\pm |\LJ(\pi)|$ for some semi-simple $\oFl$-representation
of $D^{\times}$. The symbol $(\sigma^{\rm ss}(\pi),L(\pi))$ denotes
the (transposed) Weil-Deligne
$\oFl$-representation  associated to $\pi$ by the Vign\'eras
correspondence of \cite[Thm 1.8.2]{VigLanglands}.
This is the Zelevinski-like normalization of the local Langlands correspondence
mod $\ell$.  
Therefore, to put it in english words, the above theorem offers a geometric
interpretation of the nilpotent part\footnote{Note that, in contrast to the $\ell$-adic setting, 
 this nilpotent part has no obvious arithmetic interpretation, in the
 sense that it cannot be related to any infinitesimal action of the
$\ell$-inertia of $W_{K}$.}
 of this Vign\'eras correspondence,
at least for those {unipotent} representations such that $\LJ(\pi)\neq
0$.


\medskip

Let us say a few words about the proof of the theorem.
 Note first that, since $\LJ(\pi)$ is most
often $0$, we are soon reduced to the case when $\pi$ is a
subquotient of the smooth representation $\ind{B}{G}{\oFl}$ induced
from the trivial representation of some Borel subgroup.
In section \ref{sec:ellipt-princ-ser} we classify these
subquotients in a suitable way, making thereby explicit the
corresponding block of the decomposition matrix of $G$, and
we compute the associated Weil-Deligne and $D^{\times}$ representations.
In section \ref{sec:cohomology-complex}, we study the unipotent summand
of the cohomoloy complex.
In particular, thanks to our congruence
hypothesis,  we may split it in a non-trivial way according to
weights.
Note that, in principle, all this study can be carried out in a purely local way,
using Yoshida's model of the tame Lubin-Tate space. However, for
reference convenience, we invoke at some point
Boyer's description of the cohomology of the whole tower in
\cite{Boyer2}, the proof of which uses global arguments.
An alternative argument uses
the  Faltings-Fargues isomorphism of \cite{FarFal}.
Then in section \ref{sec:proof-main-theorem} we prove the theorem, by
some fairly explicit computations. 

One crucial ingredient is that we
easily, and without any computation, get a complete description of
$(R_{\pi}^{*},L_{\pi}^{*})$ for $\pi$ the trivial representation,
thanks to the properties of the Gross-Hopkins
period map. The theorem above is expected to hold true for \emph{any}
smooth irreducible $\oFl$-representation $\pi$ under the congruence
hypothesis, but we are still missing some control on the pair $(R_{\pi}^{*},L_{\pi}^{*})$
when $\pi$ is a general Speh representation.

\section{Elliptic principal series}\label{sec:ellipt-princ-ser}

By definition, an irreducible smooth $\oFl$-representation is called
\emph{elliptic} if it is not a linear combination of proper
parabolically induced representations. Note that by \cite[Thm 3.1.4]{jlmodl},
this is equivalent to $\LJ(\pi)\neq 0$. 
According to \cite[Cor 3.2.2]{jlmodl}\footnote{where unfortunately the
  term elliptic has a slightly different meaning}, any elliptic
principal series is
an unramified twist of a subquotient of the
induced representation $\ind{B}{G}{\oFl}$ for some Borel
subgroup $B$. The converse is not true in general, 
but it is true under the Coxeter congruence relation,
as we will see below.


\subsection{Parametrization and decomposition matrix}

\alin{A reminder on the $\ell$-adic case} \label{reminder}
We denote by $B$ the subgroup of upper triangular matrices in $G$ and
by $S$ the
set of simple roots of the diagonal torus $T$ in $\hbox{Lie}(B)$.
The power set $\PC(S)$ of $S$ is in $1$-to-$1$ correspondence with the set of
parabolic subgroups containing $B$. Namely, to each subset
 $I\subset S$ is associated the unique parabolic subgroup $P_{I}$
with $\hbox{Lie}(P_{I}) = \hbox{Lie}(B)+\sum_{\alpha\in \ZM\langle I\rangle}
\hbox{Lie}(G)_{\alpha}$.
In particular we have $P_{\emptyset}=B$ and $P_{S}=G$. 

\begin{defn}
  For any ring $R$, we put $i_{I}(R):=\ind{P_{I}}{G}{R}$, and
  $$v_{I}(R):=i_{I}(R)/\sum_{J\supset I} i_{J}(R).$$
\end{defn}

Let $\delta_{B}$ denote the $R$-valued modulus character of
$B$. We assume that $R$ contains a square root of $q$ in $R$ and we
choose such a root in order to define
$\delta:=\delta_{B}^{-\frac{1}{2}}$ as well as the normalized Jacquet
functor $r_{B}$ along $B$. 
Write
$X:=X_{*}(T)\otimes \RM$, so that $S$ is naturally a subset of the
dual $\RM$-vector space of $X$. Following \cite[2.2.3]{dp}, we
associate to each subset 
$I\in\PC(S)$ a connected component 
$$X_{I}:=\{x\in X,\,\forall\alpha\in S,
\varepsilon_{I}(\alpha)\langle x,\alpha\rangle>0\}$$ 
of the complement of the union of simple root hyperplanes in $X$. Here
$\varepsilon_{I}$ is the sign function on $S$ which takes $\alpha$ to
$-1$ if and only if  $\alpha\in I$. In particular, $X_{\emptyset}$ is
the Weyl chamber associated to $B$ and $X_{S}$ is that associated to the
opposite Borel subgroup.

\begin{fac}[\cite{dp}, Lemme 2.3.3] If $R$ is a field of
  characteristic  prime to $\prod_{i=1}^{d}(q^{i}-1)$, then the following hold :
\begin{enumerate}
\item for each $I\subseteq S$, the $R$-representation  $v_{I}(R)$ is
  irreducible and we have 
$$r_{B}(v_{I}(R))=\bigoplus_{w(X_{S})\subseteq X_{I}}
   w^{-1}(\delta).$$  
\item the multiset ${\rm JH}(\ind{B}{G}{R})$ of Jordan-H\"older
  factors of $\ind{B}{G}{R}$ has multiplicity one, hence is a set, and
the map 
$$ I\in \PC(S) \mapsto v_{I}(R) \in {\rm JH}(\ind{B}{G}{R})$$
is a bijection.
\end{enumerate}
\end{fac}


Let us label  $t_{0},t_{1},\cdots,t_{d-1}$ the diagonal entries of an
element $t\in T$ (starting from the upper left corner).
We get a labeling $S = \{\alpha_{1},\cdots,\alpha_{d-1}\}$, where
$\alpha_{i}(t):= t_{i-1}t_{i}^{-1}$, and we get an
identification of the Weyl group $W$ of $T$ with the symmetric group
$\SG_{d}$ of the set $\{0,\cdots, d-1\}$. Then we see that
the condition $ w(X_{S})\subseteq X_{I}$ appearing in the
summation of point i) above is equivalent to the condition
$$I= \{\alpha_{i}\in S, \,w(i-1)<w(i)\}.$$

\alin{Classification under the Coxeter congruence relation} Here the
coefficient field is $\Fl$ or $\oFl$ and we
assume that the multiplicative order of $q$ modulo $\ell$ is exactly
$d$. Denote by $\nu_{G}$ the unramified character $g\mapsto q^{-{\rm
    val}({\rm det})(g)}$. Observe that $\nu_{G}$ is trivial on the
center of $G$ and generates a cyclic subgroup $\langle\nu_{G}\rangle$ of order $d$ of the
group of $\oFl$-valued characters of $G$. 

We put $\wt{S}:=S\cup\{\alpha_{0}\}$ where $\alpha_{0}$ denotes the
opposite of the longest root of $T$ in $\hbox{Lie}(B)$. 
Thus, if the diagonal entries of $t\in T$ are $t_{0},t_{1},\cdots,
t_{d-1}$ as above, then
$\alpha_{0}(t)=t_{d-1}t_{0}^{-1}$.
Note that $\wt{S}$ is stable under the action of the Coxeter element
$c$ of $W=\SG_{d}$ which takes $i<d-1$ to $i+1$ and $d-1$ to $0$. 
In fact $\wt{S}$ is a principal homogeneous set under the cyclic subgroup
$\langle c\rangle$  of order $d$ generated by $c$. Therefore,
it is convenient to identify $\{0,\cdots, d-1\}$ with
$\ZM/d\ZM$ through the canonical bijection, so that we simply have
$$ c(\alpha_{i})=\alpha_{i+1},\,\, \forall i\in\ZM/d\ZM.$$

We denote by $\PC'(\wt{S})$ the set of \emph{strict} subsets of
$\wt{S}$. 
  For any $I\in \PC'(\wt{S})$ we can thus choose an $i\in
  \wt{S}\setminus I$. The 
  translated subset $c^{-i}(I)$ is then contained in $S$ so we can consider
  the representation $i_{c^{-i}I}(\oFl)\otimes \nu_{G}^{i}$.

\begin{lem} \label{lemiI}
Up to semisimplification, the representation
  $i_{c^{-i}I}(\oFl)\otimes \nu_{G}^{i}$  is independent of the choice of
  $i$ in $\wt{S}\setminus I$.
We denote by $[i_{I}]$ its class in the Grothendieck group $\RC(G,\oFl)$.
\end{lem}
\begin{proof}
Up to translation by a power of $c$, 
we may assume that $I\subseteq S$, so that we have to
compare $i_{I}(\oFl)$ with $i_{c^{-i}I}(\oFl)\otimes\nu_{G}^{i}$
(assuming that $i\notin I$).
Note that the parabolic subgroups $P_{I}$ and $P_{c^{-i}I}$ are
associated. More precisely, 
any element of the normalizer of $T$ in $G$ which projects to $c^{i}$
will conjugate
the Levi component
$M_{c^{-i}I}$ to $M_{I}$.
Therefore, 
\cite[Lemme 4.13]{finitude} shows
 that in the Grothendieck group $\RC(G,\oFl)$ we have the equality
$[i_{c^{-i}I}(1)]=[i_{I}(\gamma)]$ with
$\gamma=(\delta_{P_{I}}c^{i}(\delta_{P_{c^{-i}I}}^{-1}))^{\frac{1}{2}}$.  Thus
we have to prove that $\gamma={\nu_{G}^{-i}}_{|M_{I}}$. Since restriction
to $T$ is injective on characters of $M_{I}$, we may restrict both
sides to $T$.
Using that  ${\delta_{P_{I}}}_{|T}=\delta_{B}\delta_{B\cap
  M_{I}}^{-1}$,  we get that $\gamma_{|T}=
(\delta_{B}.c^{i}(\delta_{B}^{-1}))^{\frac{1}{2}}$. 

For $k=0,\cdots, d-1$, consider the smooth character of $T$ defined by
$\varepsilon_{k}(t)= q^{-{\rm val}(t_{k})}$ where $t_{0},\cdots,
t_{d-1}$ are the diagonal entries of $t\in T$. Then we have
$$ \gamma_{|T}= \left(\prod_{k<l} \varepsilon_{k}\varepsilon_{l}^{-1}
  \prod_{k<l}
  \varepsilon_{k+i}^{-1}\varepsilon_{l+i}\right)^{\frac{1}{2}} =
\prod_{k< i\leq l} \varepsilon_{k}\varepsilon_{l}^{-1} =
\prod_{k< i}\varepsilon_{k}^{d-i}\prod_{i\leq l}\varepsilon_{l}^{-i}.$$
Of course in the first equality, the indices $k+i$ and $l+i$ should
be read modulo $d$. To get the second equality, we observe that 
for $0\leq k< l<d$, we have $k+i>l+i (\hbox{mod } d) \Leftrightarrow
k<d-i\leq l \Leftrightarrow  l+i < i \leq k+i$.
Now using the fact that $\varepsilon_{k}^{d}=1$, we get
$\gamma_{|T}=\prod_{k}\varepsilon_{k}^{-i}= (\nu_{G}^{-i})_{|T}$ as desired.
\end{proof}

\begin{rak}
  The particular case $I=\emptyset$, $i=1$ of the above lemma tells us
  that $[\ind{B}{G}{\oFl}]=[\ind{B}{G}{\oFl}\otimes\nu_{G}]$. So the
  twisting action of the cyclic group $\langle\nu_{G}\rangle$ on the
  set of classes of irreducible representations preserves the multiset
  ${\rm JH}(\ind{B}{G}{\oFl})$.
\end{rak}

We now want to isolate a certain irreducible constituent of
$[i_{I}]$. 
We follow the Zelevinski approach via degenerate Whittaker models, 
as in Vign\'eras' work. 
First we associate a partition $\lambda_{I}$ of $d$ to $I$ in the following way.
As in the previous proof, $I$ determines a conjugacy class of Levi
subgroups, namely that of $M_{c^{-i}I}$ for any $i\in \wt{S}\setminus
\{i\}$. This conjugacy class corresponds to a partition $\mu_{I}$ of
$d$, and we let $\lambda_{I}$ be the transpose of $\mu_{I}$.
For example,  $\lambda_{\emptyset}=(d)$, and
$\lambda_{I}=(1,\cdots, 1)$ whenever $|I|=d-1$.

We refer to \cite[III.1]{Vig} for the basics on
 the theory of derivatives and to \cite[V.5]{Viginduced} for the
 notion of degenerate Whittaker models.

 \begin{fac} For $I\in\PC'(\wt{S})$,
   the representation $[i_{I}]$ has a unique irreducible constituent
   $\pi_{I}$ admitting a $\lambda_{I}$-degenerate Whittaker model. 
   Moreover any other irreducible constituent has $\lambda$-degenerate
   Whittaker models only for $\lambda< \lambda_{I}$.
 \end{fac}

 \begin{rak}
   By the Lemma above, for any $I\in \PC'(\wt{S})$ we have $\pi_{c^{i}I}\simeq
   \pi_{I}\otimes\nu_{G}^{i}$. In particular for any $i\in\wt{S}$, we have 
$ \pi_{\wt{S}\setminus\{i\}}=\nu_{G}^{i}$. On the other hand
$\pi_{\emptyset}$ is the only generic constituent of
$\ind{B}{G}{\oFl}$. 
 \end{rak}

Before proceeding, we introduce some more notation. Similarly to the
banal case, we consider the complement in $X$ of the union of all hyperplanes
attached to the roots in $\wt{S}$. Its connected components are
labeled by \emph{proper} subsets of $\wt{S}$, and given by 
$$ \wt{X}_{I}:=\{x\in X,\,\forall\alpha\in \wt{S},
\varepsilon_{I}(\alpha)\langle x,\alpha\rangle>0\}$$ 
where $\varepsilon_{I}$ is the sign function attached to $I$ as
before. Note that $\wt{X}_{\wt{S}}=\wt{X}_{\emptyset}=\emptyset$ and that
$\wt{X}_{S}=X_{S}$ is again the opposite Weyl chamber to $B$. However
for $I$ a strict subset of $S$, we have $\wt{X}_{I}\neq X_{I}$.

\begin{pro}\label{proppiI}
  \begin{enumerate}\item The multiset
${\rm JH}(\ind{B}{G}{\oFl})$ is a set (multiplicity one).
\item 
    The map  $I\in \PC'(\wt{S})\mapsto \pi_{I}\in
    {\rm JH}(\ind{B}{G}{\oFl})$ is a bijection.
  \item 
For all $I\in \PC'(\wt{S})$
    the following equality holds in $\RC(G,\oFl)$ :
$$ [i_{I}]=\sum_{J\supseteq I} [\pi_{J}].$$
\item $\pi_{\emptyset}$ is a cuspidal representation, and if $I\neq
  \emptyset$ then 
$$ r_{B}(\pi_{I})= \bigoplus_{w(\wt{X}_{S})\subset \wt{X}_{I}} w^{-1}(\delta).$$

\end{enumerate}
\end{pro}
\begin{proof}
  i) Suppose $\pi$ is a non-cuspidal irreducible subquotient of
  $\ind{B}{G}{\oFl}$. Let $P=M_{P}U_{P}$ be a parabolic subgroup such that
  $\pi_{U_{P}}\neq 0$. Since $\ell$ is banal for $M_{P}$, the Mackey formula
  (or ``geometric lemma'') shows that in fact $\pi_{U_{B}}\neq 0$. But
  the congruence relation and the Mackey formula imply that
  $\ind{B}{G}{\oFl}_{U_{B}}$ has the multiplicity one property, as a
  representation of $T$. More precisely, we have
  $r_{B}(\ind{B}{G}{\oFl})=\bigoplus_{w\in W} w(\delta)$, and
  $\delta=\nu_{G}^{\frac{1-d}{2}}\prod_{i=0}^{d-1}\varepsilon_{i}^{i}$ 
  is a $W$-regular character since $q$ has order $d$.
Hence $\pi$ occurs with multiplicity one in
  $\ind{B}{G}{\oFl}$. Now any cuspidal representation is generic, so
  there is at most one  cuspidal subquotient
  of $\ind{B}{G}{\oFl}$.

ii) 
This follows from the proof of \cite[Prop
  3.2.4]{jlmodl}. However the latter reference rests on
  Vign\'eras' classification \cite[V.12]{Viginduced}, so in particular on
a difficult result of  Ariki's on the classification of simple modules
of Hecke-Iwahori algebras at roots of unity.
In fact, in our context the latter can be avoided  and replaced
by the more elementary partial classification of \cite[2.17]{VigLuminy}.
 Nevertheless, for the convenience of the reader, we sketch a complete
and more direct proof.

Let us first show the injectivity of the map.
Let $\pi$ be some irreducible subquotient of 
$\ind{B}{G}{\oFl}$, and let $\lambda=\lambda_{\pi}$ be the partition
of $d$ obtained from $\pi$ by taking successive higher derivatives.
Hence $\lambda_{1}$ is the
order of the highest non-zero derivative of $\pi$, $\lambda_{2}$ is that of
the derivative $\pi^{(\lambda_{1})}$, etc. The partition $\lambda$ is
the greatest element in the set of all partitions $\lambda'$ such that
$r_{\lambda'}(\pi)$ admits a generic subquotient. 
Here $r_{\lambda'}$ denotes the normalized Jacquet functor associated
to the standard parabolic
subgroup $P_{\lambda'}=U_{\lambda'}M_{\lambda'}$ associated to $\lambda'$.
Let $\tau$ denote a generic subquotient of $r_{\lambda}(\pi)$. We can write
$\tau = \gamma(A_{1})\otimes\cdots \otimes \gamma(A_{|\lambda|})$
where $A_{1}\sqcup \cdots \sqcup A_{|\lambda|}= \{0,\cdots, d-1\}$ is a
set-theoretical partition with $|A_{i}|=\lambda_{i}$, and for any
subset $A\subset\{0,\cdots, d-1\}$, $\gamma(A)$ denotes the unique
generic subquotient of the normalized induction $\times_{a\in
  A}{\nu}^{a-\frac{d-1}{2}}$ (so this is a representation of $\GL_{|A|}(K)$).
If $\pi=\pi_{I}$ for some $I$, then $\lambda=\lambda_{I}$ and a computation shows that 
for $k=1,\cdots, |\lambda|$,
\begin{eqnarray*}
A_{k}  & = &\big\{a\in \{0,\cdots, d-1\}, \{\alpha_{a},c^{-1}\alpha_{a},\cdots,
c^{2-k}\alpha_{a}\}\subset I\big\} \\  &  = &
 \{a\in \ZM/d\ZM, \{\alpha_{a},\alpha_{a-1},\cdots,
\alpha_{a-k+2}\}\subset I\}.
\end{eqnarray*}
 In particular, the following holds :
\begin{enumerate}
\item[a)] for all $k=1,\cdots, |\lambda|-1$, we have $A_{k+1}\subset c(A_{k})=A_{k}+1$
\item[b)] $I=\{\alpha_{i}\in\wt{S},\, i \notin A_{1}\}$.
\end{enumerate}
Hence we see that 
$\pi_{I}$ determines $I$,
so that the map in point ii) is injective.

In order to prove the surjectivity, it is enough to prove that
$\ind{B}{G}{\oFl}$ has at most $2^{d}-2$ irreducible non cuspidal constituents.
If $\pi$ is such a constituent, there is a Borel subgroup $B'$ such
that $\pi$ is the \emph{unique} irreducible quotient of the
normalized induced representation $\ip{B'}{G}{\delta}$. However, the
same argument as in \cite[2.5.4]{dp} shows that if both the chambers $C(B')$ and
$C(B'')$ are contained in a component $\wt{X}_{I}$, then the canonical
interwining operator $\ip{B'}{G}{\delta}\To{}\ip{B''}{G}{\delta}$ is
an isomorphism. Indeed, we may assume as in loc.cit that $B'$ and
$B''$ are adjacent, with wall associate to some root $r$. Then the
representation theory for $\GL_{2}$ (note that $\ell$ is banal w.r.t
to $\GL_{2}(K)$) tells us that the canonical
intertwining operator is invertible unless $q^{l(r)}=q^{\pm 1}$ in
$\oFl$, where $l(r)$ is the height of the root. With our congruence
hypothesis and the general inequality $l(r)\leq n-1$, this implies
that $l(r)$ is either $1$, $-1$ or $n-1$, which is equivalent to $\pm
r\in\wt S$.
This gives the desired bound.

iii)
By ii), we only have to show that if $J$ is any other strict subset of
$\wt{S}$, then $\pi_{J}$ occurs in $[i_{I}]$ if and only if $J\supseteq I$.
Start with  $J\supseteq I$ and choose
$i\in \wt{S}\setminus J$. Then  we have
$i_{c^{-i}J}(\Fl)\otimes\nu_{G}^{i}\subset
i_{c^{-i}I}(\Fl)\otimes\nu_{G}^{i}$, so $\pi_{J}$ occurs in $[i_{I}]$.
Conversely, suppose that $\pi_{J}$
occurs in $[i_{I}]$. 
Assume first that  $J\cup I\neq \wt{S}$ and choose $i\in
\wt{S}\setminus (J\cup I)$.
Then we see that $[\pi_{J}\otimes \nu_{G}^{-i}]$ occurs in $i_{c^{-i}I}(\Fl)\cap i_{c^{-i}J}(\Fl)=i_{c^{-i}(I\cup
  J)}(\Fl)$. Hence $\lambda_{J}\leq\lambda_{I\cup J}$, so
$\lambda_{J}=\lambda_{I\cup J}$ and finally  $I\cup
J=J$, as desired. Assume now that $J\cup I= \wt{S}$, choose
 $j\in J\setminus I$ and set $J^{*}:=J\setminus\{j\}$.
We get that  
$[\pi_{J}\otimes\nu_{G}^{-j}]$ occurs in $i_{c^{-j}I}(\Fl)\cap
i_{c^{-j}J^{*}}(\Fl)=i_{c^{-j}(I\cup   J^{*})}(\Fl)= i_{\wt{S}\setminus
  \{0\}}=\pi_{\wt{S}\setminus \{0\}}$. Hence
$\pi_{J}=\pi_{\wt{S}\setminus \{j\}}$, so $J=\wt{S}\setminus\{j\}$
which is impossible by definition of $j$.

iv) We proved in point ii) that $\ind{B}{G}{\oFl}$ has exactly $2^{d}-2$
non-cuspidal irreducible subquotients. But we constructed $2^{d}-1$
constituents, so $\ind{B}{G}{\oFl}$ has exactly one cuspidal
subquotient. We know it is generic, so it is, by definition,  $\pi_{\emptyset}$.
Now, fix some \emph{proper} subset $I$ of $\wt{S}$. Again by the proof
of the surjectivity in point ii), there is a unique proper subset
$J$ such that 
$$ r_{B}(\pi_{I})=\bigoplus_{w(\wt{X}_{S})\subset \wt{X}_{J}}
w^{-1}(\delta).$$
We still have to prove that $I=J$. Note that the condition
$w(\wt{X}_{S})\subset \wt{X}_{J}$ is equivalent to 
$$ J=\{\alpha_{j}\in \wt{S},\, wc^{-1}(j)<w(j)\}.$$
Now, items a) and b) in the proof of the injectivity in ii) show that
$I\subseteq J$. Since the map  $I\mapsto J$ is a bijection, it has to
be the identity.

\end{proof}

\alin{The decomposition matrix for elliptic representations}
Recall that an admissible smooth $\oQl$-representation $\pi$ of $G$ is called $\ell$-integral if
it contains a $G$-stable $\oZl$-lattice. Then it is known that the
reduction to $\oFl$ of
such a lattice only depends  on $\pi$ up to semisimplification, see
\cite[II.5.1.b]{Vig}. We
denote by $r_{\ell}(\pi)$ the semisimple $\oFl$-representation thus obtained.

\begin{pro}\label{decomp}
  Let $I\subseteq S$. Then we have 
$$r_{\ell}(v_{I}(\oQl))= \cas{[v_{I}(\oFl)]}{[\pi_{I}]+[\pi_{I\cup\{0\}}]}{I\neq S}{{[\pi_{S}]}}{I=S}$$
\end{pro}
\begin{proof}
Since parabolic induction commutes with inductive limits, we have
$i_{I}(R)\simeq i_{I}(\ZM)\otimes R$ for any ring $R$.
By its definition as a quotient  $v_{I}(R)=i_{I}(R)/\sum_{J\supset
    I}i_{J}(R)$, we also have $v_{I}(R)=v_{I}(\ZM)\otimes R$. 
Now, by \cite[Coro 4. 5]{SS1}, we know that
  $v_{I}(\ZM)$ is free over $\ZM$. The first equality follows.

If $I=S$, we have $v_{S}(\oFl)=\pi_{S}=\oFl$ (trivial representation),
so the second equality is clear in this case. Assume $I\neq S$. By
\cite[Prop 6.13]{SS1}, the following simplical complex is exact :
\begin{equation}
 0\To{} i_{S}(\ZM)\To{} \cdots \To{} \bigoplus_{J\supset I, |J|=|I|+1} 
i_{J}(\ZM)\To{} i_{I}(\ZM)\To{} v_{I}(\ZM)\To{} 0.\label{complexe}
\end{equation}
Since it consists of free $\ZM$-modules, it remains exact after base
change to $\oFl$.
Thus we get the equality
$$ [v_{I}(\oFl)]= \sum_{S \supseteq J\supseteq I} (-1)^{|J\setminus I|}
[i_{J}(\oFl)] $$
in $\RC(G,\oFl)$.
On the other hand, Proposition \ref{proppiI} iii) provides us with the
equality
\begin{equation}
 [\pi_{I}] = \sum_{\wt{S}\supset J \supseteq I} (-1)^{|J\setminus I|}
[i_{J}].\label{pii}
\end{equation}
Thus we get 
$$ [\pi_{I}]-[v_{I}(\oFl)]=\sum_{\wt{S}\supset J \supseteq I\cup\{0\}} (-1)^{|J\setminus I|}
[i_{J}]= -[\pi_{I\cup\{0\}}].$$

Alternatively, one could have used point iv) in
Proposition~\ref{proppiI}, and the easy fact that for any $I\subset
S$, we have $X_{I}=\wt{X}_{I}\cup\wt{X}_{I\cup\{0\}}$.

\end{proof}

\subsection{Corresponding representations}

\alin{Langlands-Jacquet transfer} We refer to \cite{jlmodl} for the
definition of the Langlands-Jacquet transfer map
$\LJ_{\oFl}:\,\RC(G,\oFl)\To{}\RC(D^{\times},\oFl)$, which is induced
by carrying  Brauer characters through the usual bijection between
regular elliptic conjugacy classes of $G$ and $D^{\times}$. We will
need the $\oFl$-valued unramified character $\nu_{D}:d\mapsto
q^{-{\rm val}\circ{\rm Nrd}(d)}$ of $D^{\times}$.

\begin{pro}\label{LJexplicite}
  For any strict subset $I\subset \wt{S}$ we have 
$$\LJ_{\oFl}(\pi_{I})=(-1)^{|I|} \sum_{j\in \wt{S}\setminus I}[\nu_{D}^{j}].$$
\end{pro}
\begin{proof}
  Since the map $\LJ_{\oFl}$ kills all parabolically induced
  representations \cite[Thm 3.1.4]{jlmodl}, equality (\ref{pii}) shows that
$$\LJ_{\oFl}(\pi_{I})=(-1)^{|\wt{S}\setminus I|+1} \sum_{j\in
  \wt{S}\setminus I} \LJ_{\oFl}(\pi_{\wt{S}\setminus\{j\}}).$$
On the other hand
$\pi_{\wt{S}\setminus\{j\}}=\nu_{G}^{j}=r_{\ell}(v_{S}(\oQl))\otimes
\nu_{G}^{j}$. By compatibility of the $\LJ$ maps with 
 reduction modulo $\ell$ \cite[Thm 1.2.3]{jlmodl} and with torsion by characters, we get that 
$\LJ_{\oFl}(\pi_{\wt{S}\setminus\{j\}})=(-1)^{|S|} [\nu_{D}^{j}]$.

\end{proof}

\alin{Different operations on Weil-Deligne representations} Before we
proceed to a description of the Galois-type representations attached
to the $\pi_{I}$'s, we need to make precise some formal properties of
Weil-Deligne representations. 

It is convenient to work in a fairly general setting. So let $\CC$ be
an essentially small, artinian, noetherian, abelian category and  let
$\CC^{\rm ss}$ be the full subcategory of semisimple
objects. The Jordan-H\"older theorem yields a map
$$ {\rm Ob}(\CC)_{/\sim} \To{} K_{+}(\CC),\,\, V\mapsto [V]$$
from the set of isomorphism classes of objects to the free monoid on
simple objects. This map induces a bijection ${\rm Ob}(\CC^{\rm
  ss})_{/\sim} \simto K_{+}(\CC)$.

Assume further that $\CC$ is endowed with an automorphism $V\mapsto
V(1)$ and denote by $V\mapsto V(n)$ its $n^{\rm th}$ iteration. Consider the
category $\NC(\CC)$ with objects all pairs $(V,N)$ with $V\in{\rm
  Ob}(\CC)$ and $N:V\To{} V(-1)$ a nilpotent morphism. With the obvious
notion of morphisms, $\NC(\CC)$ is an artinian, noetherian, abelian
category. The formalism of Deligne's filtration \cite[(1.6)]{DelWeil2} yields a map
$$ {\rm Ob}(\NC(\CC))_{/\sim} \To{} K_{+}(\CC)^{(\NM)},\,\, (V,N)\mapsto
[V,N]$$
where the RHS is the set of almost zero sequences of elements in
$K_{+}(\CC)$. Namely, we put $[V,N]:=([P_{-n}^{N}(V)])_{n\in\NM}$
where $P^{N}_{i}$ is the primitive part of the $i$-graduate of
Deligne's filtration attached to $N$.
We leave the reader check the following fact.
\begin{lem}
  The map $(V,N)\mapsto [V,N]$ induces a bijection ${\rm
    Ob}(\NC(\CC^{\rm ss}))_{/\sim} \simto K_{+}(\CC)^{(\NM)}$.
\end{lem}
As a consequence, one gets :
\begin{itemize}
\item[--] a ``semi-simplification'' process ${\rm Ob}(\NC(\CC))_{/\sim}\To{}
  {\rm Ob}(\NC(\CC^{\rm ss}))/\sim$.
\item[--] a ``transposition'' process ${\rm Ob}(\NC(\CC^{\rm ss}))_{/\sim} \simto
{\rm Ob}(\LC(\CC^{\rm ss}))_{/\sim}$, where $\LC(\CC)$ denotes the category of
pairs $(V,L)$ with $L: V\To{} V(1)$ nilpotent.
\item[--] a map ${\rm Ob}(\NC({\CC'}^{\rm ss}))_{/\sim}\To{}{\rm
    Ob}(\NC({\CC}^{\rm ss}))_{/\sim}$ for any map $K_{+}(\CC')\To{} K_{+}(\CC)$.
\end{itemize}
As an example of application, let $\CC=\Rep_{\oFl}(W_{K})$,
resp. $\CC'=\Rep_{\oQl}(W_{K})$, be the category 
 of finite dimensional representations of $W_{K}$ with $\oFl$,
 resp. $\oQl$ coefficients.
In this paper, a \emph{Weil-Deligne $\oFl$-representation} is an object of
$\NC(\CC^{\rm ss})$ (so our convention is that the Weil part of a WD
representation is semisimple).
Applying the last item to the decomposition map $r_{\ell}:\,
K_{+}(\Rep_{\oQl}(W_{K})) \To{} K_{+}(\Rep_{\oFl}(W_{K}))$ 
we get a reduction process
$$   (\sigma^{\rm ss}, N) \mapsto r_{\ell}(\sigma^{\rm
  ss}, N)= (r_{\ell}\sigma^{\rm ss}, N) $$
for Weil-Deligne representations.

\alin{The Zelevinski-Vign\'eras correspondence} According to
\cite[Thm 1.6]{VigLanglands}, there is a unique map 
$$\pi\mapsto \sigma^{\rm
  ss}(\pi),\,\,\;\Irr{\oFl}{G}\To{}\left\{\hbox{$d$-dimensional semi-simple
  $\oFl$-reps of }W_{K}\right\}_{/\sim} $$
which is compatible with the $\ell$-adic semi-simple Langlands correspondence via reduction modulo $\ell$ in the following
sense : \emph{if $\pi$ is a constituent of $r_{\ell}(\wt\pi)$ for $\wt\pi\in
\Irr{\oQl}{G}$, then $\sigma^{\rm ss}(\pi)=r_{\ell}(\sigma^{\rm
  ss}(\wt\pi))$.} In \cite{ltmodl}, we gave a geometric realization of
this map, as well as another proof of its existence.

Using her classification \`a la Zelevinski, Vign\'eras explained in 
\cite[1.8]{VigLanglands} that the above
semi-simple Langlands correspondence extends uniquely  to a \emph{bijection} :
$$\application{}{\Irr{\oFl}{G}}{\left\{\hbox{$d$-dimensional 
    Weil-Deligne $\oFl$-reps of }W_{K}\right\}_{/\sim}}{\pi}{\sigma^{Z}(\pi)=(\sigma^{\rm ss}(\pi),
N^{Z}(\pi))}
$$
such that the following compatibility with the $\ell$-adic Langlands
correspondence via reduction modulo $\ell$
holds : \emph{if $\pi$ is a constituent of $r_{\ell}(\wt\pi)$ for $\wt\pi\in
\Irr{\oQl}{G}$, and \emph{if $\lambda_{\pi}=\lambda_{\wt\pi}$}, then 
$\sigma^{Z}(\pi)=r_{\ell}(\sigma(Z(\wt\pi)))$.}

Here, $Z$ denotes the Zelevinski involution for
$\oQl$-representations, and  the precise meaning of $r_{\ell}$ in the context of WD
representations was explained in the last paragraph. Further,
$\lambda_{\pi}$ is the partition of $d$ attached to
 $\pi$ by taking successive higher non-zero  derivatives, as in the proof
 of point ii) of Proposition \ref{proppiI}. Note that the mere
 existence of a $\wt\pi$ fulfilling the conditions above is highly non
 trivial in general, and rests on Ariki's work on cyclotomic Hecke algebras.

Our aim in this paper is to provide a (partial) geometric
interpretation of this enhanced correspondence, by means of a Lefschetz
operator. Therefore we will focus on the ``transposed'' WD
representation, as defined in the previous paragraph :
$$ (\sigma^{\rm ss}(\pi),L(\pi)):= {^{t}(\sigma^{\rm ss}(\pi),
{N^{Z}(\pi)})}.$$

We now want to compute explicitly these transposed WD representations
for the elliptic principal series. 
This will involve the $\oFl$-character $\nu_{W}:\,w\mapsto q^{-{\rm val}({\rm
    Art}_{K}^{-1}(w))}$ where ${\rm Art}_{K}$ is the local class field
homomorphism with takes a uniformizer to a geometric Frobenius.
For simplicity, we will use the
so-called ``Hecke normalization'' of
Langland's correspondence. 

\begin{pro} \label{WDexplicit}
  For any strict subset $I\subset \wt{S}$, we have 
$\sigma^{\rm ss}(\pi_{I})\simeq \bigoplus_{i=0}^{d-1} \nu^{i}_{W}$ and
in a good eigenbasis, $L(\pi_{I})$ is given by the matrix
$\sum_{\alpha_{i}\in I}E_{i-1,i}$
\end{pro}
\begin{proof}
  The correspondence is compatible with twisting in the sense that
  $\sigma^{Z}(\pi\otimes\nu_{G})= \sigma^{Z}(\pi)\otimes\nu_{W}$. Since our
  proposed solution is also compatible with twisting, we may assume
  that $I\subset S$. In this case we know that $\pi_{I}$ appears in
  $r_{\ell}(v_{I}(\oQl))$. We also know that
  $\lambda_{\pi_{I}}=\lambda_{v_{I}(\oQl)}=\lambda_{I}$. Therefore we
  have $\left(\sigma^{\rm ss}(\pi_{I}), L(\pi_{I})\right)=r_{\ell}\left(\sigma^{\rm
    ss}(v_{I}(\oQl)), L(v_{I}(\oQl))\right)$. But the latter was
computed in \cite[3.2.4]{lefschetz}.

\end{proof}


\subsection{Computation of some Ext groups}

This section is rather technical in nature and should be skipped at
first reading. We first check that some computations of Ext
groups between the $v_{J}$'s and the $i_{I}$'s 
performed by Orlik in \cite{Orlikext} remain valid in our present
context, although Orlik's hypotheses are not satisfied. Then we
proceed to compute  Ext groups between the $\pi_{J}$'s and $i_{I}$'s.

\alin{Context and notation}
We fix a uniformizer $\varpi$ of $K$ and we will consider Yoneda
extensions in the category $\Mo{\oFl}{G/\varpi^{\ZM}}$ of smooth
$\oFl$-representations of $G/\varpi^{\ZM}$. 
Recall that a subset $I\subseteq S$ determines a standard parabolic subgroup
$P_{I}$, the standard Levi component of which is denoted by $M_{I}$. We
also denote by $W_{I}$ the Weyl group of $T$ in $M_{I}$, which is also
 the subgroup of $W$ generated by reflections associated to roots in
 $I$. We define a $\oFl$-vector space
$$ Y_{I}:= X^{*}(M_{I}/Z(G))\otimes_{\ZM}\oFl $$
where $X^{*}$ denotes the group of $K$-rational characters and $Z$
means ``center''.

Symbols $r_{P}$ and $i_{P}$ will stand for normalized parabolic
functors along the parabolic subgroup $P$ and $\delta_{P}$ will denote the
modulus character of $P$. With this notation we
have e.g. $i_{I}(\oFl)=i_{P_{I}}(\delta_{P_{I}}^{-\frac{1}{2}})$. We will
also put $\delta=\delta_{B}^{-\frac{1}{2}}$. 
Finally, the symbol $\EC xp(T,\sigma)$ denotes the set of characters of $T$
occuring as subquotients of the admissible $\oFl T$-representation $\sigma$.

\begin{lem}\label{lemext}
Let $I$ be a strict subset of $S$. 
\begin{enumerate}
\item If $\pi$, $\pi'$ are two principal series of $M_{I}$, then
  $$\left(W_{I}\cdot\exp(T,r_{B\cap M_{I}}(\pi))\cap
  W_{I}\cdot\exp(T,r_{B\cap M_{I}}(\pi'))=\emptyset\right) \Rightarrow
  \ext{*}{\pi}{\pi'}{M_{I}/\varpi^{\ZM}}=0.$$
\item $\ext{*}{\oFl}{\oFl}{M_{I}/\varpi^{\ZM}}=\bigwedge^{*}Y_{I}$
\end{enumerate}
\end{lem}
\begin{proof}
  i) The assumption means that $\pi$ and $\pi'$ have disjoint cuspidal
  supports. Since $\ell$ is banal for $M_{I}$, the vanishing of Ext
  follows from \cite[6.1]{Vigext}.

ii) The argument in \cite[Prop. 9]{Orlikext} shows that
$\ext{*}{\oFl}{\oFl}{M_{I}/\varpi^{\ZM}}=\ext{*}{\oFl}{\oFl}{M_{I}/M_{I}^{0}\varpi^{\ZM}}$
where $M_{I}^{0}$ is the subgroup of $M_{I}$ generated by compact
elements (note that $\ell$ is prime to the pro-index
$[M_{I}^{0}:[M_{I},M_{I}]]$). Since $\ell$ is also prime to the
torsion in the abelian group $M_{I}/M_{I}^{0}\varpi^{\ZM}$, we know
that 
$\ext{*}{\oFl}{\oFl}{M_{I}/M_{I}^{0}\varpi^{\ZM}}
=\bigwedge^{*}(\Hom_{\rm gps}(M_{I}/M_{I}^{0}\varpi^{\ZM},\oFl))
=\bigwedge^{*}(\Hom_{\rm gps}(M_{I}/M_{I}^{0}\varpi^{\ZM},\ZM)\otimes_{\ZM}\oFl)
$.
Finally, the usual map $\chi\mapsto {\rm val}_{K}\circ \chi$ yields an
isomorphism $X^{*}(M_{I}/Z(G))\To{} \Hom_{\rm
  gps}(M_{I}/M_{I}^{0}\varpi^{\ZM},\ZM)$.
\end{proof}

\begin{rak}\label{gradedmodule}
  A consequence of item ii) of the foregoing lemma and Frobenius reciprocity
 is that for any representation
  $\pi$ of $G/\varpi^{\ZM}$, the graded space
  $\ext{*}{\pi}{i_{I}(\oFl)}{G/\varpi^{\ZM}}\simeq
  \ext{*}{(\pi)_{U_{P_{I}}}}{\oFl}{M_{I}/\varpi^{\ZM}}$
 is naturally a graded right module
     over the graded algebra $\bigwedge^{*}Y_{I}$.
In particular there is a canonical graded map 
$$ \hom{\pi}{i_{I}(\oFl)}{G/\varpi^{\ZM}} \otimes_{\oFl} 
\bigwedge^{*}Y_{I}
\To{} \ext{*}{\pi}{i_{I}(\oFl)}{G/\varpi^{\ZM}}.$$
This map is clearly functorial in $\pi$. It is also functorial in $I$
in the sense that if $J\subset I$ we have a commutative diagram
$$\xymatrix{ \hom{\pi}{i_{I}(\oFl)}{G/\varpi^{\ZM}} \otimes_{\oFl} 
\bigwedge^{*}Y_{I} \ar[r] \ar[d] &
\ext{*}{\pi}{i_{I}(\oFl)}{G/\varpi^{\ZM}} \ar[d] \\
 \hom{\pi}{i_{J}(\oFl)}{G/\varpi^{\ZM}} \otimes_{\oFl} 
\bigwedge^{*}Y_{J}
\ar[r] &  \ext{*}{\pi}{i_{J}(\oFl)}{G/\varpi^{\ZM}}
}
$$
where the vertical maps are induced by the inclusion
$i_{I}(\oFl)\injo i_{J}(\oFl)$ and the restriction map $Y_{I}\To{}Y_{J}$.

\end{rak}

\begin{prop}\label{extiIiJ}
Let $I,J$ be two subsets of $S$, with $I$ a \emph{strict} subset. Then
the canonical map
$$ \hom{i_{J}(\oFl)}{i_{I}(\oFl)}{G/\varpi^{\ZM}} \otimes_{\oFl}  \bigwedge^{*}Y_{I}
\To{} \ext{*}{i_{J}(\oFl)}{i_{I}(\oFl)}{G/\varpi^{\ZM}}$$
is an isomorphism. In other words, we have
$$ \ext{*}{i_{J}(\oFl)}{i_{I}(\oFl)}{G/\varpi^{\ZM}}\simeq 
\left\{  
  \begin{array}{ll}
    \bigwedge^{*} Y_{I} & \hbox{if } J\supseteq I \\
0 & \hbox{otherwise}
  \end{array}
\right..
$$
Moreover, the natural map 
$\ext{*}{i_{K}(\oFl)}{i_{I}(\oFl)}{G/\varpi^{\ZM}}\To{}
\ext{*}{i_{J}(\oFl)}{i_{I}(\oFl)}{G/\varpi^{\ZM}}$ is an isomorphism for any $J\supseteq K \supseteq I$.
\end{prop}
\begin{proof}
We follow \cite[Prop. 15]{Orlikext} but we avoid Lemma 16 of
\emph{loc. cit.} which might fail to be true in our context. By Frobenius reciprocity, we have 
$$\ext{*}{i_{J}(\oFl)}{i_{I}(\oFl)}{G/\varpi^{\ZM}} =
\ext{*}{r_{P_{I}}\circ i_{P_{J}}(\delta_{P_{J}}^{-\frac{1}{2}})}
{\delta_{P_{I}}^{-\frac{1}{2}}}{M_{I}/\varpi^{\ZM}} $$
and by the geometric Mackey formula, $r_{P_{I}}\circ i_{P_{J}}(\delta_{P_{J}}^{-\frac{1}{2}})$ has a
filtration with graded pieces of the form $Q_{w}:=i_{M_{I}\cap w(P_{J})}
\left(w(\delta_{P_{J\cap w^{-1}(I)}}^{-\frac{1}{2}})\right)$,
 where $w$ runs over all elements
in  $W$ such that $w(J)\subset \Phi^{+}$ and $w^{-1}(I)\subset \Phi^{+}$ (this is a
complete set of representatives of double cosets in $W_{I}\ba W/W_{J}$). 
Using again the geometric Mackey formula we get
$$ W_{I}\cdot \exp\left(T,r_{B\cap M_{I}}(Q_{w})\right) = 
W_{I}\cdot \exp\left(T,r_{B\cap M_{I\cap w(J)}}(w(\delta_{P_{J\cap
      w^{-1}(I)}}^{-\frac{1}{2}}) )\right) =  
W_{I}\cdot \{w(\delta)\}.
$$
On the other hand, we have 
$$W_{I}\cdot \exp\left(T,r_{B\cap M_{I}}(\delta_{P_{I}}^{-\frac{1}{2}}
  )   \right) =  W_{I}\cdot \{\delta\}. $$
Since $\delta$ is $W$-regular, item i) of the above Lemma tells us that 
$\ext{*}{Q_{w}}{\delta_{P_{I}}^{-\frac{1}{2}}}{M_{I}/\varpi^{\ZM}}=0$
unless $w\in W_{I}$. In this case, we must have $w=1$ 
 so that  $Q_{w}=Q_{1}= i_{M_{I}\cap P_{J}}(\delta_{P_{I\cap J}}^{-\frac{1}{2}})$ is the top
quotient of the geometric Mackey filtration and the canonical map 
$$ \ext{*}{i_{M_{I}\cap P_{J}}(\delta_{P_{I\cap J}}^{-\frac{1}{2}})}{\delta_{P_{I}}^{-\frac{1}{2}}}
{M_{I}/\varpi^{\ZM}} \To{} \ext{*}{i_{J}(\oFl)}{i_{I}(\oFl)}{G/\varpi^{\ZM}} $$
is an isomorphism. Using Casselman's reciprocity, the LHS identifies
with 
$$\ext{*}{\delta_{P_{I\cap J}}^{-\frac{1}{2}}}
{r_{M_{I}\cap \overline{P_{J}}}(\delta_{P_{I}}^{-\frac{1}{2}})}
{M_{I\cap J}/\varpi^{\ZM}}
=
\ext{*}{\delta_{P_{I\cap J}}^{-\frac{1}{2}}}
{\delta_{P'_{I\cap J}}^{-\frac{1}{2}}}
{M_{I\cap J}/\varpi^{\ZM}}
$$
where $\overline{P_{J}}$ is the opposite parabolic subgroup to $P_{J}$ w.r.t
$M_{J}$ and $P'_{I\cap J}$ is the semistandard parabolic subgroup with
Levi component $M_{I\cap J}$ and unipotent
radical $U_{I}(\overline{U_{J}}\cap M_{I})$. Let $B'$ be the Borel subgroup
with unipotent radical $U_{I}(\overline{U_{J}}\cap M_{I})(U_{\emptyset}\cap
M_{I\cap J})$. Point i) of the previous Lemma tells us that 
the RHS of the last displayed
formula vanishes unless there is $w\in W_{I\cap J}$ such that $w(B)=B'$.
But then $w(B)\cap M_{I\cap J}= B'\cap M_{I\cap J}$ hence $w=1$, thus
$P'_{I\cap J}=P_{I\cap J}$ which is possible only if $J\supseteq I$.

Eventually we have proved the desired vanishing when $J$ does not contain
$I$, and we have proved that if $J\supseteq I$, the canonical map 
$$ \ext{*}{\delta_{P_{I}}^{-\frac{1}{2}}}{\delta_{P_{I}}^{-\frac{1}{2}}}
{M_{I}/\varpi^{\ZM}} \To{} \ext{*}{i_{J}(\oFl)}{i_{I}(\oFl)}{G/\varpi^{\ZM}} $$
is an isomorphism. We conclude the computation
using item ii) of the previous Lemma. The last assertion follows from
the functorial nature of the above map.
\end{proof}

\ali The complex (\ref{complexe}) yields a spectral sequence
$$ E^{pq}_{1}= \bigoplus_{K\supseteq J, |K\setminus J|=p}
\ext{q}{i_{K}(\oFl)}{i_{I}(\oFl)}{G/\varpi^{\ZM}} \Rightarrow
\ext{p+q}{v_{J}(\oFl)}{i_{I}(\oFl)}{G/\varpi^{\ZM}} $$
whence in particular an edge map
\begin{equation}
\ext{*}{\oFl}{i_{I}(\oFl)}{G/\varpi^{\ZM}}\To{}
\ext{|S\setminus J|+*}{v_{J}(\oFl)}{i_{I}(\oFl)}{G/\varpi^{\ZM}}.\label{edge}
\end{equation}
Thanks to the last Proposition, the same argument as \cite[Prop. 17]{Orlikext}
gives the following expression.
\begin{cor} \label{calculext1}
Let $I,J$ be subsets of $S$ with $I$ a strict subset.
  \begin{enumerate}\item If  $I\cup J\neq S$ then 
$\ext{*}{v_{J}(\oFl)}{i_{I}(\oFl)}{G/\varpi^{\ZM}}=0$.
\item If $I\cup J=S$, then the map (\ref{edge}) is an isomorphism, so
  we get an isomorphism
   $$\ext{*}{v_{J}(\oFl)}{i_{I}(\oFl)}{G/\varpi^{\ZM}}\simeq
  \bigwedge^{*-|S\setminus J|} Y_{I}.$$
Moreover, if $I'$ is another strict subset of $S$ which contains $I$,
then the natural map $\ext{*}{v_{J}(\oFl)}{i_{I'}(\oFl)}{G/\varpi^{\ZM}} \To{}
\ext{*}{v_{J}(\oFl)}{i_{I}(\oFl)}{G/\varpi^{\ZM}}$ is induced by the
natural restriction map $Y_{I'}\To{} Y_{I}$.
 \end{enumerate}
\end{cor}

\begin{rak}\label{remext}
We may recast the foregoing corollary by stating that the canonical
map
$$\ext{|S\setminus J|}{v_{J}(\oFl)}{i_{I}(\oFl)}{G/\varpi^{\ZM}}
\otimes_{\oFl} \bigwedge^{*}Y_{I} \To{}
\ext{*+|S\setminus J|}{v_{J}(\oFl)}{i_{I}(\oFl)}{G/\varpi^{\ZM}}$$
is an isomorphism, and that 
$\ext{|S\setminus J|}{v_{J}(\oFl)}{i_{I}(\oFl)}{G/\varpi^{\ZM}}\simeq
\oFl$ if $J\cup I=S$ and is zero otherwise.  
\end{rak}

Next we turn to extensions between the $\pi_{J}$'s and the $i_{I}$'s.

\begin{prop}\label{calculext2}
  Let $J$ be a strict subset of $\wt{S}$, and
  $I$ a strict subset of $S$. 
  \begin{enumerate}
  \item If $0\in J$, then $\ext{*}{\pi_{J}}{i_{I}(\oFl)}{G/\varpi^{\ZM}}=0$.
  \item Otherwise, the natural map is an isomorphism
$$\ext{*}{v_{J}(\oFl)}{i_{I}(\oFl)}{G/\varpi^{\ZM}} \simto
\ext{*}{\pi_{J}}{i_{I}(\oFl)}{G/\varpi^{\ZM}}.$$ 
  \end{enumerate}
\end{prop}
\begin{proof}
  Note first that ii) follows from i) since
  $[v_{J}(\oFl)]=[\pi_{J}]+[\pi_{J\cup\{0\}}]$.
Now, in order to prove i) we first use Frobenius reciprocity to get 
$$\ext{*}{\pi_{J}}{i_{I}(\oFl)}{G/\varpi^{\ZM}} =
\ext{*}{r_{P_{I}}(\pi_{J})}
{\delta_{P_{I}}^{-\frac{1}{2}}}{M_{I}/\varpi^{\ZM}}. $$
By Proposition \ref{proppiI} iv), we have
$ \exp\left(T,r_{B\cap M_{I}}(r_{P_{I}}(\pi_{J}))\right) = 
\left\{w^{-1}(\delta),\,\, w(\wt X_{S})\subset \wt X_{J}\right\}.$
Since $\exp(T,r_{B\cap
  M_{I}}(\delta_{P_{I}}^{-\frac{1}{2}}))=\{\delta\}$ and since
$\delta$ is $W$-regular, 
Lemma \ref{lemext} shows that 
we are left to prove that $\left\{w\in W,\,\, w(\wt X_{S})\subset \wt
  X_{J}\right\}\cap W_{I}=\emptyset$.
Now, identifying $W$ with $\SG_{d}$ as in paragraph \ref{reminder}, 
the condition $w(\wt X_{S})\subset \wt X_{J}$ is equivalent to 
the condition  $J=\{\alpha_{j}\in \wt{S},\, wc^{-1}(j)<w(j)\}$, so in
particular it implies the property $w(n-1)<w(0)$. However, since $I$ is
proper, this property is never satisfied by some $w\in W_{I}$.
\end{proof}

\section{The cohomology complex}\label{sec:cohomology-complex}

In this section, we focus on the useful part of the cohomology
complex, namely on that which pertains to the unipotent block of the
category of smooth $\oZl$-representations.

\subsection{The unipotent block}
According to Vign\'eras \cite[IV.6.2]{Viginduced}, the category
$\Mo{\oFl}{G}$ is a product of indecomposable Serre subcategories
called ``blocks''. This product of blocks corresponds to the
partition of the set of irreducible $\oFl$-representations according
to the ``inertia class of supercuspidal support''. 
Among them, the \emph{unipotent block} is by
definition the one which contains the trivial representation. 
In representation theory of finite groups, this would be rather called
the ``principal block''. 
Here we want to lift this block to
$\oZl$-representations. Note that the usual way of lifting idempotents
via Hensel's lemma is not adapted to the $p$-adic case, since Hecke algebras
are not finitely generated modules over $\oZl$.
Therefore, we will exhibit a progenerator of the desired block.
In all this subsection, no congruence assumption on the pair
$(q,\ell)$ is required.

\def\bu{\mathbf{b}}
\def\eu{\mathbf{e}}

\alin{Unipotent blocks for a finite $\GL_n$}
For a finite group of Lie type $\bar{G}$, we will denote by
$\bu_{\bar{G}}$ the central idempotent in the group algebra
$\Zl[\bar{G}]$ which cuts out the direct sum of all blocks which
contain a \emph{unipotent} $\oQl$-representation (in the sense of Deligne-Luzstig).


\begin{lem} \label{lemfinite}
Let $\bar{P}=\bar{M}\bar{U}$ be a parabolic subgroup of $\bar{G}$, and let
  $e_{\bar{U}}$ be the idempotent associated to the $p$-group
  $\bar{U}$. Then  we have
$    e_{\bar{U}}\bu_{\bar{G}}=e_{\bar{U}}\bu_{\bar{M}}=\bu_{\bar{M}}e_{\bar{U}}.$
\end{lem}
\begin{proof}
According to \cite{BroueMichel}, an irreducible $\oQl$-representation $\pi$ satisfies
$\bu_{\bar{G}}\pi\neq \{0\}$ \emph{if and only if} it belongs to the
Deligne-Lusztig 
series associated to some semi-simple conjugacy class in the dual group
$\bar{G}^{*}$ which consists of $\ell$-elements. We call such a
representation $\ell$-unipotent. In this case, all irreducible
subquotients of $\pi_{U}$ are $\ell$-unipotent representations of
$M$. Indeed, this follows by adjunction from the ``dual'' statement that, 
if $\sigma$ is an $\ell$-unipotent representation of
$M$ then all irreducible subquotients of $\ind{P}{G}{\sigma}$ are
$\ell$-unipotent, see \cite[Cor. 6]{Lusznilpfin}. This shows that, denoting by
 $\bu_{\bar{G}}':=1-\bu_{\bar{G}}$  the complementary
 idempotent,
 we have $\bu_{\bar{M}}'
 e_{\bar{U}}\bu_{\bar{G}}=0$
 and $\bu_{\bar{M}} e_{\bar{U}}\bu_{\bar{G}}'=0$.
Then we get
$e_{\bar{U}}\bu_{\bar{G}}=  (\bu_{\bar{M}}'+\bu_{\bar{M}})
e_{\bar{U}}\bu_{\bar{G}} = \bu_{\bar{M}}
e_{\bar{U}}\bu_{\bar{G}} =\bu_{\bar{M}}
e_{\bar{U}}(1-\bu_{\bar{G}}')=\bu_{\bar{M}}
e_{\bar{U}}.  $
\end{proof}

\begin{fac} \label{facfinite}
Assume $\bar{G}=\GL_{n}(\FM_{q})$. Then
an irreducible $\oFl$-represen\-tation
$\bar\pi$ of 
$\bar{G}$ satisfies $\bu_{\bar G}\bar\pi\neq 0$ if and only if  it is a subquotient of
$\ind{\bar{B}}{\bar G}{\oFl}$.  
\end{fac}
\begin{proof} Any subquotient of $\ind{\bar{B}}{\bar G}{\oFl}$
  occurs in the reduction of a unipotent irreducible $\oQl$-representation, hence
  belongs to the category cut out by  $\bu_{\bar G}$.
Conversely, fix $\bar\pi$ such that $\bu_{\bar G}\bar\pi\neq 0$. We may assume that $\bar\pi$ is cuspidal, since for
  $\bar{P}=\bar M\bar U$ a parabolic subgroup such that $\bar\pi_{\bar
    U}\neq 0$ we also have $\bu_{\bar M}(\bar\pi_{\bar U})\neq 0$ (as
  in the previous proof). But
  then
  in terms of the Dipper-James classification, $\bar\pi$ is of the
  form $D(s,1)$ for some elliptic semi-simple $\ell$-element of
  $\bar{G}^{*}=\bar{G}$, see \cite[Coro 5.23]{Dipper}. Thus, in terms
  of the James-Dipper classification, it is also of the form
  $D(1,(n))$, see \cite[Thm 5.1]{DJ}, which means that $\bar\pi$ is
  the only non-degenerate subquotient in $\ind{B}{G}{\oFl}$.
\end{proof}

\alin{Construction of the block}
Here we put $\bar{G}=\GL_{d}(\FM_{q})$.
We may  view $\mathbf{b}_{\bar{G}}$ as a central idempotent of the
$\Zl$-algebra $\HC_{\Zl}(\GL_{d}(\OC))$ 
of locally constant distributions on $\GL_{d}(\OC)$.
Then we put
$$P_{\mathbf{b}}:=\cind{\GL_{d}(\OC)}{G}{\mathbf{b}_{\bar G}\HC_{\Zl}(\GL_{d}(\OC))}$$
and we define
$\Mo{\mathbf{b}}{G}$ as the full subcategory of $\Mo{\Zl}{G}$
consisting of all objects $V$ that are generated by $\mathbf{b}_{\bar G} V$ over $\Zl G$.

We will use similar notation to denote somewhat more familiar objects
; letting $\eu_{\bar G}$ be the idempotent attached to the pro-$p$-radical
of $\GL_{d}(\OC)$, we also put
$$P_{\eu}:=\cind{\GL_{d}(\OC)}{G}{\mathbf{e}_{\bar G}\HC_{Zl}(\GL_{d}(\OC))}$$
and we define the category $\Mo{\eu}{G}$ as above. We recall the
following result, which is a special case
of ``level decomposition'', see e.g.  \cite[App. A]{finitude}.

\begin{fac}
  The category $\Mo{\mathbf{e}}{G}$ is a direct factor of
  $\Mo{\Zl}{G}$ and is pro-generated by $P_{\mathbf{e}}$. In
  particular, there is an idempotent $\eu$ of the center of the
  category $\Mo{\Zl}{G}$ such that for any object $V$ we have $\eu
  V=\sum_{g\in G/\GL_{d}(\OC)} g \eu_{\bar G} V$.
\end{fac}

Now we can state the main result of this subsection.

\begin{pro}
  The category $\Mo{\mathbf{b}}{G}$ is a direct factor of
  $\Mo{\eu}{G}$ and is pro-generated by $P_{\mathbf{b}}$. It consists of all objects $V$, all irreducible subquotients
  of which are not annihilated by $\bu_{\bar G}$.
\end{pro}
\begin{proof}
  From its definition, $P_{\mathbf{b}}$ clearly is a generator of the category
  $\Mo{\mathbf{b}}{G}$, and is a finitely generated projective object
  of $\Mo{\oZl}{G}$. 

Let us prove that $\Mo{\mathbf{b}}{G}$ is a Serre subcategory. For
this, we will apply the general result of \cite[Thm 3.1]{MS1}. 
For any vertex $x$ of the semisimple building of $G$, we denote by
$G_{x}$ its stabilizer, $G_{x}^{+}$ the pro-$p$-radical of its
stabilizer and $\bar{G}_{x}:=G_{x}/G_{x}^{+}$ the reductive
quotient. We thus get an idempotent $\bu_{x}\in \HC_{\oZl}(G_{x})$ by
inflation from $\bu_{\bar{G_{x}}}$. If $g\in G$, we clearly have $\bu_{gx}=g\bu_{x}g^{-1}$.
Therefore, to apply \cite[Thm 3.1]{MS1} we are left to check the two following
properties (\emph{cf} \cite[Def 2.1]{MS1}):
\begin{enumerate}
\item $\bu_{x}\bu_{y}=\bu_{y}\bu_{x}$ for any adjacent vertices $x,y$.
\item $\bu_{x}\bu_{z}\bu_{y}=\bu_{x}\bu_{y}$ whenever $z$ is
  adjacent to $x$ and belongs to
  the convex simplicial hull of $\{x,y\}$.
\end{enumerate}
Note first  that the definition of $\bu_{x}$ extends to any
facet $F$ of the building. 
Further, let $\mathbf{e}_{F}:=e_{G_{F}^{+}}$ denote
the idempotent associated to the pro-$p$-group $G_{F}^{+}$. We know
that properties i) and ii) are satisfied by the system
$(\eu_{x})_{x}$, and more precisely we have
$\eu_{x}\eu_{y}=\eu_{[x,y]}$ whenever $x$ and $y$ are adjacent vertices.
Therefore, the above lemma shows that
$\bu_{x}\mathbf{e}_{y}=\bu_{x}\mathbf{e}_{[x,y]}=\bu_{[x,y]}$
and thus $\bu_{x}\bu_{y}=\bu_{[x,y]}=\bu_{y}\bu_{x}$.
As for property ii),  starting from
$\eu_{x}\eu_{y}=\eu_{x}\eu_{z}\eu_{y}$, we get
$\bu_{x}\bu_{y}=\bu_{x}\eu_{z}\bu_{y}= \bu_{x}\bu_{z}\bu_{y}$, as desired.

We now know that $\Mo{\bu}{G}$ is a Serre subcategory of the 
Serre subcategory $\Mo{\eu}{G}$ cut out by the system
$(\eu_{x})_{x}$. For a vertex $x$, define
$\bu'_{x}:=\eu_{x}-\bu_{x}$, which is lifted from the idempotent
$1-\bu_{\bar{G_{x}}}$ of $\oZl[\bar{G_{x}}]$. The same argument as above shows that the
system $(\bu'_{x})_{x}$ satisfies properties i) and ii) and therefore
cuts out a Serre subcategory of $\Mo{\eu}{G}$, which is easily seen
to be a complement to $\Mo{\bu}{G}$. Therefore the latter is a
direct factor in $\Mo{\eu}{G}$.
The last statement
of the proposition is clear.
\end{proof}

\begin{no}
  We will denote by $\bu$ the idempotent of the center of the category
  $\Mo{\Zl}{G}$ which projects a representation $V$ to its largest subobject
  $\bu V$ in $\Mo{\bu}{G}$. Concretely, we have  $\bu V = \sum_{g\in
    \GL_{d}(K)/\GL_{d}(\OC)} g.\bu_{\bar G}V$.
\end{no}


\begin{prop}\label{irredunipbloc}
  A representation $\pi\in\Irr{\oFl}{G}$ belongs to $\Mo{\bu}{G}$ if
  and only if it is an irreducible subquotient of some
  $\ind{B}{G}{\chi}$ with $\chi$ an unramified character of $B$.
In particular, $\Mo{\bu}{G}\cap\Mo{\oFl}{G}$ is Vign\'eras' unipotent
block \cite[IV.6.3]{Viginduced}.
\end{prop}
\begin{proof}
Let $\eu_{\bar G}\in\HC_{\oZl}(\GL_{d}(\OC))$ be the 
  idempotent associated to the kernel of the reduction map
  $\GL_{d}(\OC)\To{}\GL_{d}(\FM_{q})$.
  By Mackey formula, the residual representation of $\bar{G}$ on
  $\eu_{\bar G}\ind{B}{G}{\chi}$ is isomorphic to
  $\ind{\bar{B}}{\bar{G}}{\oFl}$ with obvious notation. Since 
$\ind{B}{G}{\chi}$ belongs to the level zero subcategory $\Mo{\eu}{G}$, so does each one of its irreducible
subquotients. Hence for such a subquotient $\pi$,  $\eu_{\bar G}\pi$ is a
non-zero subquotient of $\ind{\bar{B}}{\bar{G}}{\oFl}$, so
that $\bu_{\bar G}\pi\neq 0$ and $\pi$ belongs to $\Mo{\bu}{G}$.

Conversely, let $\pi$ be an irreducible $\oFl$-representation such
that $\bu\pi\neq 0$. Choose a parabolic subgroup $P=MU$ and a
supercuspidal $\oFl$-representation $\sigma$ of $M$ such that $\pi$ occurs as a
subquotient of $\ind{P}{G}{\sigma}$. 
As above,  Mackey formula tells us that 
$\eu_{\bar G}\ind{P}{G}{\sigma}\simeq \ind{\bar P}{\bar G}{\eu_{\bar
    M}\sigma}$, 
with  obvious notation.
So by Lemma \ref{lemfinite} we get
$\bu_{\bar G}\ind{P}{G}{\sigma}\simeq \ind{\bar P}{\bar G}{\bu_{\bar
    M}\sigma}$ and finally $\bu_{\bar M}\sigma\neq 0$.
By \cite[3.15]{Vig},  we know that $\sigma$ is of the form
$\cind{M\cap\GL_{n}(\OC)}{M}{\bar\sigma}$ for some \emph{super}cuspidal
$\oFl$-representation $\bar\sigma$ of the Levi subgroup $\bar M$ of
$\bar G$, image of $M\cap \GL_{n}(\OC)$ by the projection to $\bar
G$. Here, supercuspidal is equivalent to the fact that the semisimple
elliptic class $s$ associated to $\bar\sigma$ consists of $\ell'$-elements.
However, an easy computation \cite[3.14]{Vig} shows that $\bu_{\bar
  M}\bar\sigma=\bu_{\bar M}\sigma$.
Therefore $\bu_{\bar M}\bar\sigma$ is non
zero and $s$ consists of $\ell$-elements by definition. Hence $s=1$,
or equivalently, $M$ is a torus and $\bar\sigma$
the trivial representation of $\bar M$.
\end{proof}

\begin{rak}
  In terms of the Langlands correspondence, the irreducible
  $\oFl$-representa\-tions $\pi$ of the principal/unipotent block are
  those such that $\sigma(\pi)^{\rm ss}$ is a sum of unramified
  characters. This formulation might extend to other $p$-adic groups,
  as suggested by the finite field picture.
\end{rak}


\subsection{The complex}

In the first two paragraphs of this subsection, no congruence hypothesis
on the pair $(q,\ell)$ is required. From paragraph
\ref{descriptioncoho} on, we will work under the Coxeter congruence relation.

\alin{The tower and its cohomoly complexes}
We refer to \cite{StrauchAdv} or \cite[3.1]{lt} for the definition of
the Lubin-Tate space $\mltn$ of height $d$ and level $n$, which we see
as a  $\breve{K}$-analytic space, endowed with a continuous action of
$D^{\times}$, an action of $\GL_{d}(\OC/\varpi^{n}\OC)$ and a Weil descent datum
to $K$.  Although in this paper we will be mainly interested in the tame level
$\MC_{\rm LT, 1}$, the formalism used to define the complex requires
the whole ``tower'' $(\mltn)_{n\in \NM}$ and in particular the action
of $G=\GL_{d}(K)$ which can be defined on this tower. Maybe the most precise way to
describe this action is to introduce the category $\NM(G)$ with
set of objects $\NM$ and arrows given by $\Hom(n,m):=\{g\in G, g
M_{d}(\OC) g^{-1}\subset \varpi^{m-n}M_{d}(\OC)\} $ and to note that
the $\mltn$'s are the image of a functor from $\NM(G)$ to the category
whose objects are  $\breve{K}$-analytic spaces with continuous action of $D^{\times}$
and Weil descent datum to $K$, and morphisms are finite \'etale
equivariant morphisms. This allows one to define the complex
$$ R\Gamma_{c}:= R\Gamma_{c}(\mlt^{\rm ca},\Zl) \in 
D^{b}(\Rep^{\infty,
  c}_{\Zl}(G\times D^{\times}\times W_{K}))$$
as in \cite[3.3.3]{lt}. Let us note that the diagonal subgroup
$K^{\times}$ of $G\times D^{\times}$ acts trivially on the tower hence
also on the cohomology.

It is technically important to recall that the tower is ``induced'' from
a ``sub-tower'' denoted $(\mltn^{(0)})_{n\in \NM}$ which is stable
under the subgroup 
$$ (GDW)^{0}:= \left\{(g,\delta,w)\in G\times D^{\times}\times W_{K},
|\det(g)|^{-1}|{\rm Nr}(\delta)||{\rm Art}_{K}(w)|=1\right\}.$$
So we have a complex 
$ R\Gamma_{c}^{(0)}:=R\Gamma_{c}(\mlt^{\rm ca,(0)},\Zl) \in 
D^{b}(\Rep^{\infty,  c}_{\Zl}(GDW)^{0})$
together with an isomorphism \cite[(3.5.2)]{lt}
$$ R\Gamma_{c}\simeq \cInd{(GDW)^{0}}{GDW}
 R\Gamma_{c}^{(0)}.$$
An important consequence of this is the following compatibility with
twisting. For any smooth character $\chi$ of $K^{\times}$ and  any
representation $\pi$ of $G$, we have
\begin{equation}
 R_{(\chi\circ\det)\otimes\pi} \simeq
(\chi\circ ({\rm Nr}\cdot {\rm Art}_{K}))\otimes R_{\pi} 
\,\,\,\hbox{ in } \,D^{b}(\Rep_{\Zl}(D^{\times}\times W_{K})).
\label{twist}
\end{equation}

We need yet another variant. Let us fix a uniformizer $\varpi$ of
$K$ and see it as a central element of $G$. Its action on the tower is
free (it permutes the connected components), so we may consider the
quotient tower $(\mltn/\varpi^{\ZM})_{n\in\NM}$ and its cohomology
complex 
$$R\Gamma_{c,\varpi}:= R\Gamma_{c}(\mlt^{\rm ca}/\varpi^{\ZM},\Zl) \in 
D^{b}(\Rep^{\infty,
  c}_{\Zl}(G/\varpi^{\ZM}\times D^{\times} \times W)).$$ We then have
isomorphisms (\emph{cf} \cite[3.5.3]{lt})
$$R\Gamma_{c,\varpi} \simeq 
R\Gamma_{c} \otimes^{L}_{\Zl[\varpi^{\ZM}]}\Zl \simeq
\cInd{(GDW)^{0}\varpi^{\ZM}}{GDW}
 R\Gamma_{c}^{(0)}.$$
Because of the first isomorphism, if $\pi$ is a representation on which $\varpi$ acts
trivially, then
$R_{\pi}\simeq\Rhom_{G/\varpi^{\ZM}}(R\Gamma_{c,\varpi},\pi)$. Since
any irreducible representation may be twisted to achieve this condition
$\pi(\varpi)=1$, we see that we don't loose any generality in
restricting attention to $R\Gamma_{c,\varpi}$.

\alin{The tame part}
We take up the notation $\eu$, $\eu_{\bar G}$ of the previous
subsection and denote by $\HC_{\eu}$ the commuting algebra $\endo{\Zl
  G}{P_{\eu}}$, which  identifies
with the Hecke algebra of compactly
supported $\Zl$-valued  $(1+\varpi M_{d}(\OC))$-bi-invariant measures
on $G$.

The complex $\eu_{\bar G}R\Gamma_{c}(\mlt^{\rm ca},\Zl)$ is naturally
an object of $D^{b}(\Rep_{\HC_{\eu}}^{\infty}(D^{\times}\times
W_{K}))$ and we recover the direct summand $\eu R\Gamma_{c}(\mlt^{\rm
  ca},\Zl)$ via the usual equivalence of categories. Namely we have,
as in \cite[Lemme 3.5.9]{lt},
$$ \eu R\Gamma_{c}(\mlt^{\rm ca},\Zl) \simeq  P_{\eu}
\otimes^{L}_{\HC_{\eu}} \eu_{\bar G}R\Gamma_{c}(\mlt^{\rm ca},\Zl).$$
Now, if we restrict the action to $\GL_{d}(\OC)$, we have
by construction an isomorphism in $D^{b}(\Rep^{\infty,c}_{\Zl}(\bar G
\times D^{\times}\times W_{K}))$.
$$ \eu_{\bar G}R\Gamma_{c}(\mlt^{\rm ca},\Zl) \simto
R\Gamma_{c}(\MC_{\rm LT, 1}^{\rm ca},\Zl) \simeq \cInd{\bar G \times
  (DW)^{0}}{\bar G\times D^{\times}\times W_{K}} R\Gamma_{c}(\MC_{\rm
  LT, 1}^{\rm ca, (0)},\Zl).$$

The tame Lubin-Tate space $\MC_{\rm LT,1}^{(0)}$ was studied by Yoshida in
\cite{yoshida}. He exibited in particular a certain affinoid subset
$\NC$ of $\MC_{\rm LT,1}^{{0}}$ which acquires good reduction over
$\breve{K}[\varpi^{1/(q^{d}-1)}]$, with special fiber equivariantly
isomorphic to the Deligne-Lusztig covering $Y(c)$ associated to the Coxeter
element of $\bar G$.
In \cite{tame}, we showed that the restriction map induces an
isomorphism $R\Gamma(\MC_{\rm  LT, 1}^{\rm ca, (0)},\Zl)\simto
R\Gamma(\NC^{\rm ca},\Zl)$. Taking duals, we thus get an
isomorphism in $D^{b}(\Rep_{\Zl}(\bar G))$
\begin{equation}
R\Gamma_{c}(Y(c),\Zl)\simto R\Gamma_{c}(\MC_{\rm LT, 1}^{\rm ca, (0)},\Zl).\label{isomDL}
\end{equation}
In particular we get the following important property.
\begin{pro}\label{torsionfree}
  The cohomology spaces of both the complexes $\eu R\Gamma_{c}$ and
  $\eu R\Gamma_{c,\varpi}$ are torsion-free. 
\end{pro}
Indeed, the torsion-freeness for $Y(c)$ follows from Lemma 3.9 and
Corollary 4.3 of \cite{BonRou2}. We emphasize the fact that no
hypothesis on the pair $(q,\ell)$ is required here.

\alin{The unipotent part :  $\ell$-adic cohomology}
From this paragraph on, we assume that the order of $q$ in
$\Fl^{\times}$ is $d$.
We take up the notation of the previous subsection, and we
consider the direct summand $\bu R\Gamma_{c}(\mlt^{\rm
  ca},\Zl)$, or rather its variant $\bu R\Gamma_{c,\varpi}$. 
There is a fairly explicit description of the $\Ql$-cohomology of
this complex. We first recall a classical construction.
  Let $\theta:\,\FM_{q^{d}}^{\times}\To{}\oQl^{\times}$ be a character which is
  ${\rm Frob}_{q}$-regular. Define :
  \begin{itemize}
  \item[--]  a representation
    $\rho(\theta):=\cInd{\OC_{D}^{\times}\varpi^{\ZM}}{D^{\times}}(\theta)$
    of $D^{\times}$, where $\OC_{D}^{\times}\varpi^{\ZM}$ acts via the reduction
    map $\OC_{D}^{\times}\To{}\FM_{q^{d}}^{\times}$.
  \item[--] a representation
    $\sigma(\theta):=\cind{I_{K}\varphi^{d\ZM}}{W_{K}}{\theta}$ where
    $I_{K}\varphi^{d\ZM}$ acts
    via the tame inertia map $I_{K}\To{}\mu_{q^{d}-1}\simeq \FM_{q^{d}-1}$.
  \item[--] a representation
    $\pi(\theta):=\cind{\GL_{d}(\OC)\varpi^{\ZM}}{G}{\pi^{0}_{\theta}}$
where $\pi^{0}_{\theta}$ is the cuspidal representation of
$\GL_{d}(\FM_{q})$ associated to $\theta$ by the Green (or
Deligne-Lusztig) correspondence.
 \end{itemize}
All these representations are irreducible and
 depend only on the ${\rm Frob}_{q}$-conju\-gacy class of $\theta$.
Moreover, they are associate by the Langlands and
Jacquet-Langlands correspondences.

\begin{fac} \label{descriptioncoho}
Let $I_{\varpi}:=\Zl[G\times D^{\times} \times W_{K}/(GDW)^{0}\varpi^{\ZM}]$.
  \begin{enumerate}\item 
    For $i=1,\cdots, d-1$, there is an isomorphism
$$\HC^{d-1+i}(\bu
    R\Gamma_{c,\varpi})\otimes\Ql \simto v_{\{1,\cdots, i\}}(\Ql)(-i) \otimes
    I_{\varpi}.$$

    \item For $i=0$, there is a (split) exact sequence
$$0\To{} \KC\To{} \HC^{d-1}(\bu
    R\Gamma_{c,\varpi})\otimes \Ql \To{} v_{\emptyset}(\Ql) \otimes
    I_{\varpi}\To{} 0$$
and an isomorphism
 $\KC\otimes\oQl\simeq \bigoplus_{\theta}
 \pi(\theta)\otimes\rho(\theta)^{\vee} \otimes \sigma(\theta)^{\vee}$,
 where $\theta$ runs over ${\rm Frob}_{q}$-conjugacy classes of ${\rm
   Frob}_{q}$-regular characters $\FM_{q^{d}}^{\times}\To{}\oZl^{\times}$ which are
 $\ell$-congruent to the trivial character. 
  \end{enumerate}
\end{fac}
\begin{proof}
The shortest argument here is to invoke Boyer's description of the
$\oQl$-cohomology of the whole Lubin-Tate tower in \cite{Boyer2} (see 
\cite[4.1.2]{lt} for an account featuring a notation consistent with that of the present paper),
together with the characterization  of irreducible objects of the
unipotent block in Proposition \ref{irredunipbloc}. We note that the
maps  $\HC^{d-1+i}(\bu R\Gamma_{c,\varpi})\otimes \Ql
\To{}v_{\{1,\cdots,i\}}(\Ql)(-i)$ are induced by the canonical morphism
\begin{equation}
 R\Gamma_{c}(\mlt^{\rm ca,(0)},\Zl)\To{} 
R\Gamma_{c}(\mlt^{\rm ca,(0)},\Zl) \otimes^{L}_{\Zl[\OC_{D}^{\times}]} \Zl.\label{coteLT1}
\end{equation}

Alternatively, if one wants to avoid Boyer's machinery,
it is possible to derive almost everything from Yoshida's construction
\cite{yoshida}, via isomorphism (\ref{isomDL}). More precisely,
put $w_{i}:= \HC^{d-1+i}(\bu R\Gamma_{c,\varpi}\otimes^{L}_{\Zl[\OC_{D}^{\times}/\varpi^{\ZM}]}\Zl)$. 
Then, by using a similar feature of Deligne-Lusztig varieties,  one
can show that the above morphism of complexes induces isomorphisms
$$ \HC^{d-1+i}(\bu R\Gamma_{c,\varpi}) \otimes \Ql \simto
w_{i}\otimes\Ql$$
for $i>0$, as well as an exact sequence
$$ \KC\injo \HC^{d-1}(\bu R\Gamma_{c,\varpi}) \otimes \Ql \twoheadrightarrow
w_{0}\otimes\Ql.$$
Further, one finds an isomorphism  $\eu_{\bar G}w_{i}\simeq \eu_{\bar
  G}(v_{\{1,\cdots, i\}}(\Zl)\otimes I_{\varpi})$.
However, what is a priori missing is enough information on Hecke operators
acting on $\eu_{\bar G}w_{i}$ in order to recognize $w_{i}$ as
isomorphic to $v_{\{1,\cdots, i\}}(\Zl)\otimes I_{\varpi}$.
One highly non trivial way to get around this problem is to
 invoke the Faltings-Fargues isomorphism
of \cite{FarFal} (see \cite[3.4]{lt} for a brief description) to
move to the so-called Drinfeld tower (see
\cite[3.2]{lt} for an overview on this tower).
Then the morphism of complexes (\ref{coteLT1}) is carried to 
\begin{equation}
 R\Gamma_{c}(\mdr^{\rm ca, (0)},\Zl)\To{} R\Gamma_{c}(\MC_{Dr,0}^{\rm ca,(0)},\Zl)\label{coteDr1}
 \end{equation}
and the right hand side is  the Drinfeld
upper half space whose cohomology is computed by combinatorics, and
shown by Schneider and Stuhler to be isomorphic to $v_{\{1,\cdots,i\}}(\Zl)$.
\end{proof}

We let $\Pi$ be a uniformizer of $D$ such that $\Pi^{d}=\varpi$, and
we fix a ``geometric'' Frobenius element $\varphi$ in $W_{K}$.
We are going to decompose the complex 
$\bu R\Gamma_{c,\varpi}$ in the category $D^{b}(\Rep_{\Zl}(G/\varpi^{\ZM}))$
according to the action of $\Pi$ and $\varphi$.
Since $K^{\times}_{\rm diag}$ acts trivially on the tower, the action
of $\Pi$ on $R\Gamma_{c,\varpi}$  is obviously killed by the
polynomial $X^{d}-1$. 
Further, as a corollary to the description above and to the
torsion-freeness result of Corollary \ref{torsionfree}, 
we get :
\begin{cor} For any integer $0\leq i \leq d-1$, the
 action of $\varphi$ on $\HC^{d-1+i}(\bu R\Gamma_{c,\varpi})$ is killed by
 the polynomial $X^{d}-q^{id}$. 
\end{cor}

\alin{The unipotent part : splitting} \label{splitting}
Put $P_{\varphi}(X):=\prod_{i=0}^{d-1 }(X^{d}-q^{id})$. By the above
corollary and \cite[Lemme A.1.4 i)]{lt}, $P_{\varphi}(\varphi)$ acts
by zero on the whole complex $\bu R\Gamma_{c,\varpi}$. The ring
$A_{\varphi}:=\Zl[X]/P_{\varphi}(X)$ is a semi-local ring, hence
decomposes as a product $A_{\varphi}=\prod_{\mG}{A_{\varphi}}_{\mG}$
of its localizations at maximal ideals. Since $q$ is a primitive
$d$-root of unity in $\Fl$ by our congruence
hypothesis, the maximal ideals of this ring are
$\mG_{i}:=(\ell,X-q^{i})$, $i=0,\cdots, d-1$ and we denote
$A_{\varphi}=\prod_{i=0}^{d-1}A_{\varphi,i}$ the associated decomposition.
Accordingly we get a decomposition \cite[Prop 1.6.8]{Neeman} 
$$ \bu R\Gamma_{c,\varpi} \simeq \bigoplus_{i=0}^{d-1} (\bu R\Gamma_{c,\varpi})_{i}
\,\,\hbox{ in } D^{b}(\Mo{\Zl}{G/\varpi^{\ZM}}).$$

Similarly, the ring $A_{\Pi}:=\Zl[X]/(X^{d}-1)$ is semi-local with maximal
ideals $(\ell,X-q^{j})$, $j=0,\cdots,d-1$ and we get a sharper
decomposition
$$ \bu R\Gamma_{c,\varpi} \simeq \bigoplus_{i,j=0}^{d-1} (\bu R\Gamma_{c,\varpi})_{i,j}
\,\,\hbox{ in } D^{b}(\Mo{\Zl}{G/\varpi^{\ZM}}).$$
Note that each one of these direct summands is preserved by the action of
$\varphi$ and $\Pi$, but not necessarily by that of $I_{K}$ and
$\OC_{D}^{\times}$. Let $\zeta$ denote the Teichm\"uller lift of $q$, \emph{i.e.}
 the only primitive $d$-root of unity in $\Zl$ which is
$\ell$-congruent to $q$. 
By construction, the action of $\Pi$ on $(\bu
R\Gamma_{c,\varpi})_{i,j}$ is by multiplication by $\zeta^{j}$, while that
of $\varphi$ is killed by the polynomial $\prod_{k}(X-q^{i-k}\zeta^{k})$.

Moreover these direct summands satisfy the following properties.
\begin{equation}
  \label{torsion}
  \begin{array}[c]{l}
  (\bu R\Gamma_{c,\varpi})_{i,j}\simeq
 \zeta^{j\val_{K}\circ\det^{-1}} (\bu R\Gamma_{c,\varpi})_{i-j,0} \\
 \hbox{with action of $\varphi$ and $\Pi$ twisted by $\zeta^{j}$}.
\end{array}
\end{equation}
This follows indeed from~(\ref{twist}).
\begin{equation}
  \label{cohofacdir}
  \begin{array}[c]{l}
 \hbox{There is a distinguished triangle }\\ 
c_{i}[0] \To{}  (\bu R\Gamma_{c,\varpi})_{i,0}[d-1] \To{} h_{i}(-i)[-i] \To{+1}
\\ \hbox{with $c_{i}$ a cuspidal $\ell$-torsion free representation} 
\\\hbox{and $h_{i}$ a $G$-invariant lattice in $v_{\{1,\cdots, i\}}(\Ql)$}.
\end{array}
\end{equation}
This follows from Fact \ref{descriptioncoho} and Proposition \ref{torsionfree}.
Note that by convention we set $\{1,\cdots, i\}=\emptyset$ if $i=0$.

Let us put $\bar h_{i}:=h_{i}\otimes \oFl$. By Proposition
\ref{decomp}, we have the following equality in the Grothendieck
group :
\begin{equation}
[\bar h_{i}]=[v_{\{1,\cdots,i\}}(\oFl)]=[\pi_{\{1,\cdots,i\}}]+[\pi_{\{0,\cdots,i\}}],\label{hi}
\end{equation}

\begin{rak} For $i\neq d-1$, it can be shown that $\bar h_{i}$ is not isomorphic to
  $v_{\{1,\cdots,i\}}(\oFl)$. More precisely,
  $v_{\{1,\cdots,i\}}(\oFl)$ is a non-split extension of
  $\pi_{\{0,\cdots,i\}}$ by $\pi_{\{1,\cdots,i\}}$, while $\bar h_{i}$
  is a non-trivial extension going the other way. The same phenomenon
  appears for the Deligne-Lusztig variety, see in particular \cite[Thm
  4.1]{Dudas1} which provides a description of the finite field
  analogue of $\bar h_{i}$. In the present contex, let us simply
  mention without proof that the morphism (\ref{coteDr1}) induces a
  map 
  \begin{equation}
    \label{coteDr2}
   H^{d-1+i}_{c}(\MC_{Dr,1}^{\rm ca},\oFl)=\bar h_{i}\To{}
   H^{d-1+i}_{c}(\MC_{Dr,0}^{\rm ca},\oFl)= v_{\{1,\cdots, i\}}(\oFl)
  \end{equation}
which is non-zero, with kernel and cokernel both isomorphic to $\pi_{\{0,\cdots,i\}}$.
Of course this map is also induced by the morphism (\ref{coteLT1}). 
\end{rak}

\section{Proof of the main theorem}\label{sec:proof-main-theorem}

Let $\pi$ be a $\oFl$-representation of $G$.
Recall the definition of the graded vector space $R_{\pi}^{*}$ from
the introduction. For convenience, we will shift this definition by
$[1-d]$, \emph{i.e.} we consider now
$$R_{\pi}^{*}:=\HC^{*}(\Rhom_{\Zl G}(R\Gamma_{c}[d-1],\pi)) .$$
This is a graded smooth $\oFl$-representation of $D^{\times}\times
W_{K}$, whose grading is supported in the range $[1-d, d-1]$ by
\cite[Prop 2.1.3]{ltmodl}.

In all this section, we work under the Coxeter congruence hypothesis,
\emph{i.e.} we assume that the order of $q$ in $\Fl^{\times}$ is $d$.

\subsection{Computation of $R_{\pi}^{*}$ for $\pi$ an elliptic
  principal series}

\alin{Preliminaries}
Assume now that $\pi$ belongs to the unipotent block and that its
central character is trivial on $\varpi$. Then we have $R_{\pi}^{*}=\HC^{*}(\Rhom_{\Zl
  (G/\varpi^{\ZM})}(R\Gamma_{c,\varpi}[1-d],\pi))$, and
according to \ref{splitting}, we may decompose it as
$$ R_{\pi}^{*}=\bigoplus_{i,j=0}^{d-1} (R_{\pi}^{*})_{i,j} ,\,\,\hbox{ where }
(R_{\pi}^{*})_{i,j}:=\HC^{*}(\Rhom_{\Zl
  (G/\varpi^{\ZM})}((\bu R\Gamma_{c,\varpi})_{i,j},\pi)).$$
Concretely, $(R_{\pi}^{*})_{i,j}$ is the intersection of the
generalized $q^{-i}$-eigenspace of $\varphi$ with the generalized $q^{-j}$-eigenspace of $\Pi$.
As already mentioned, these summands need not  be stable under the action of
$I_{K}$ and $\OC_{D}^{\times}$.
However, the description of the $\ell$-adic cohomology of $\bu
R\Gamma_{c,\varpi}$ in \ref{descriptioncoho}, together with the
the $\ell$-torsion freeness of its integral cohomology show that both
$I_{K}$ and $\OC_{D}^{\times}$ act \emph{trivially} on the
$D^{\times}\times W_{K}$ semi-simplifications
$\HC^{k}(\bu R\Gamma_{c,\varpi}\otimes^{L}_{\Zl}\oFl)^{\rm ss}$, $k\in
\NM$. Therefore, the same is true
for $R_{\pi}^{*,\rm ss}$. As a consequence,  
letting $I_{K}$ and $\OC_{D}^{\times}$ act trivially on each $(R_{\pi}^{*})_{i,j}$, 
we get the following  equality in $\RC(D^{\times}\times W_{K},\oFl)$ :
\begin{equation}
 R_{\pi}^{*,\rm ss} \simeq \bigoplus_{i,j=0}^{d-1} (R_{\pi}^{*})_{i,j}^{\rm ss}
 = \bigoplus_{i,j=0}^{d-1}
 (\nu_{_{D}}^{j}\otimes\nu_{_{W}}^{i})^{{\rm dim}_{\oFl}(R_{\pi}^{*})_{i,j}}.
\label{decompsemsimpl}
 \end{equation}

Recall also from property (\ref{twist}) that we have
$R_{\nu_{_{G}}\pi}^{*}\simeq (\nu_{_{D}}\otimes\nu_{_{W}})\otimes R_{\pi}^{*}$.
Therefore we get isomorphisms
\begin{equation}
(R_{\nu_{_{G}}\pi}^{*})_{i,j}\simeq (R_{\pi}^{*})_{i-1,j-1}.\label{twist2}
\end{equation}

The aim of this subsection is to prove  Theorem
\ref{theocomputRpi} below, which describes
explicitly each $(R_{\pi_{I}}^{*})_{i,j}$. We first introduce some
notation. 

\ali \label{defpartial} For an integer $k$ between
$0$ and $d-1$ and a subset $I$ of $S$, we put
$$\partial_{I}(k):=k-\delta(k,I)
\,\, \hbox{ where }\,\, 
\delta(k,I):=|I\cup \{1,\cdots, k\}|-|I\cap \{1,\cdots, k\}|.$$ 
These functions already appear in \cite{dp}, see in particular Lemma
4.4.1 of \emph{loc. cit.} The following property is elementary.
\begin{fac}
  The map $k\in \{0,\cdots, d-1\}\mapsto \partial_{I}(k)\in\ZM$ is
  non-decreasing, with image $\{-|I|,-|I|+2,\cdots, |I|-2, |I|\}$.
More precisely, writing $I=\{i_{1},\cdots,i_{|I|}\}$ and putting
$i_{0}:=0$ and $i_{|I|+1}:=d$, we have
$\partial^{-1}_{I}(-|I|+2j)=\{i_{j},\cdots,i_{j+1}-1\}$.
 
\end{fac}

In the next statement, we extend the function $\partial_{I}$ to $\ZM$ by making it $d$-periodic.

\begin{theo}\label{theocomputRpi}
  Let $I$ be a strict subset of $\wt{S}$ and let $i,j$ be integers
  between $0$ and $d-1$. We have
$$ (R_{\pi_{I}}^{*})_{i,j} \simeq 
\left\{  
  \begin{array}{ll}
  \oFl[\partial_{c^{-j}I}({i-j})] & \hbox{if $j\notin I$}\\
0 & \hbox{if $j\in I$}
  \end{array}
\right..
$$
\end{theo}

Since $\pi_{I}\simeq \nu_{_{G}}^{j}\pi_{c^{-j}I}$, we see that the
statement above is compatible with the twisting property
(\ref{twist2}). Therefore we only have to prove it when $j=0$.
We will treat separately the vanishing statement (when $0\in I$) and
the non-zero cases (when $0\notin I$), and we start with a special case.

\alin{The case $|I|=d-1$} \label{proofpitriv}
Here we prove Theorem \ref{theocomputRpi} for characters, \emph{i.e.}
for $|I|=d-1$. By the above remark on the effect of twisting by
$\nu_{_{G}}$, we may 
assume that $I=S$, so that
$\pi_{I}=\oFl$ is the trivial representation of $G$.
In this case, we have
$$ R_{\oFl}^{*} = \HC^{*}\left(\Rhom_{\oFl}\left(\oFl\otimes^{L}_{\oFl
      G} R\Gamma_{c},\oFl\right)[1-d]\right) .$$
By the second Lemma of paragraph (A.1.1) in \cite{ltmodl}, we have
$$ \HC^{*}\left(\oFl\otimes^{L}_{\oFl G} R\Gamma_{c}\right) \simeq
H^{*}(\PM^{d-1,\rm ca},\oFl) = \bigoplus_{i=0}^{d-1}\oFl[-2i](-i),$$
where the action of $D^{\times}$ is trivial and that of $W_{K}$ is
described by the Tate twists.
Forgetting about technicalities, this merely expresses the fact that
$G$ acts freely on the tower $(\mltn)_{n\in\NM}$ and that the quotient
is the so-called Gross-Hopkins period space, which is isomorphic to
the projective space $\PM^{d-1}$ over $\knr$.
It follows that
$$ R_{\oFl}^{*} = \bigoplus_{i=0}^{d-1} (\nu_{_{D}}^{0}\otimes \nu_{_{W}}^{i})[1-d+2i].$$
Since $\partial_{S}(i)=1-d+2i$, we have proved Theorem
\ref{theocomputRpi} for $I=S$, and thus for any $I\subset \wt{S}$ of cardinality $d-1$.

\alin{Vanishing when $j\in I$} \label{vanishing}
As already mentioned,  we may assume that $j=0$. Fix a
strict subset $I$ of $\wt{S}$ which contains $0$. 
We will prove in this paragraph that 
\begin{equation}
 \hbox{for all $i=0,\cdots, d-1$, we have $(R^{*}_{\pi_{I}})_{i,0}=0$}.\label{vansta}
  \end{equation}
 We argue by decreasing induction on $|I|$. The case $|I|=d-1$ was
  treated in \ref{proofpitriv}, so let us assume $|I|<d-1$.
Recall from  Lemma \ref{lemiI} and Proposition \ref{proppiI} iii) that 
for any $k\in \wt{S}\setminus I$, we have $c^{-k}I\subset S$ and
$$ [\nu_{_{G}}^{k}\otimes i_{c^{-k}I}(\oFl)] 
= [i_{I}] = \sum_{J\supseteq I} [\pi_{J}].$$
Therefore, using the induction hypothesis, it is enough to find a
$k\in \wt{S}\setminus I$ such that
$$ \Rhom_{\Zl(G/\varpi^{\ZM})}((\bu R\Gamma_{c,\varpi})_{i,0},\nu_{_{G}}^{k}\otimes
i_{c^{-k}I}(\oFl))=0.$$
Let us start with a random $k$ in $\wt{S}\setminus I$.
By (\ref{cohofacdir}) and Frobenius reciprocity, we have an isomorphism
$$ \HC^{*}(\Rhom_{\Zl (G/\varpi^{\ZM})}((\bu R\Gamma_{c,\varpi})_{i,0},
\nu_{_{G}}^{k}\otimes i_{c^{-k}I}(\oFl))) 
\simeq 
\ext{*+i}{\nu_{_{G}}^{-k}\otimes \bar h_{i}}{ i_{c^{-k}I}(\oFl)}{\oFl (G/\varpi^{\ZM})}.$$ 
Further, by (\ref{hi}) we have $[\nu_{_{G}}^{-k}\otimes
\bar h_{i}]=[\pi_{c^{-k}\{0,\cdots,i\}}] + [\pi_{c^{-k}\{1,\cdots,i\}}]$.
Therefore, applying Proposition \ref{calculext2} i), we get 
\begin{equation}
 \ext{*}{\nu_{_{G}}^{-k}\otimes \bar h_{i}}{
  i_{c^{-k}I}(\oFl)}{G/\varpi^{\ZM}}=0 \,\, \hbox{whenever }\, k\in\{1,\cdots,i\}.\label{implyvanish}
\end{equation}
In other words, if $k\in\{1,\cdots, i\}$, we are done. Let us thus
assume that $k\notin \{1,\cdots, i\}$.
In this case,  Proposition \ref{calculext2} ii) and 
Corollary \ref{calculext1} i) tell us that 
$$ \ext{*}{\nu_{_{G}}^{-k}\otimes \bar h_{i}}{
  i_{c^{-k}I}(\oFl)}{G/\varpi^{\ZM}}=0 \,\, \hbox{whenever }\,
c^{-k}\{1,\cdots,i\}\cup c^{-k}I\neq S.$$ 
This means that if $I\cup \{1,\cdots, i\} \neq \wt{S}\setminus\{k\}$,
we are done. In particular, if $i=0$ (in which case $\{1,\cdots,
i\}=\emptyset$ by convention), we are done, because $|I|<d-1$.
Now let us assume the contrary, \emph{i.e.}
$I\cup \{1,\cdots, i\} = \wt{S}\setminus\{k\}$ (and therefore $i\geq
1$). Again because of $|I|<d-1$, this
means that $\{1,\cdots, i\}$ contains an element $k'$ which does not
belong to $I$. Applying (\ref{implyvanish}) to this $k'$, we get
$$ \ext{*}{\nu_{_{G}}^{-k'}\otimes \bar h_{i}}{
  i_{c^{-k'}I}(\oFl)}{G/\varpi^{\ZM}}=0$$
and this finishes the proof of (\ref{vansta}).


\alin{Computation when $j\notin I$} \label{nonvanishing}
Again we may assume that
$j=0$, and hence that $I\subset S$. 
The vanishing property of \ref{vanishing} shows that the map
$\pi_{I}\injo v_{I}:=v_{I}(\oFl)$ induces isomorphisms
\begin{equation}
 (R^{*}_{\pi_{I}})_{i,0}\simto (R^{*}_{v_{I}})_{i,0} \,\, \hbox{ for }\,
i=0,\cdots, d-1\label{isompiIvI}
\end{equation}
because the cokernel $v_{I}/\pi_{I}$ is isomorphic to $\pi_{I\cup\{0\}}$.

Now, we will use the exact sequence (\ref{complexe}) in order to
compute $(R^{*}_{v_{I}})_{i,0}$. It provides us with a spectral
sequence 
$$ E_{1}^{pq}= 
\bigoplus_{ \tiny
  \begin{array}{l}
S\supseteq J\supseteq I   \\ |J\setminus  I|=-p
\end{array}
} 
(R^{q}_{i_{J}})_{i,0} \Rightarrow (R^{p+q}_{v_{I}})_{i,0} $$
where we have abbreviated $i_{J}:=i_{J}(\oFl)$. 
A priori, this spectral sequence vanishes outside the range $ -|S\setminus
I|\leq p \leq 0$ and $q\geq -i $.
Its differential
$d_{1}$ has degree $(1,0)$, and is given by the natural maps 
$(R^{*}_{i_{J'}})_{i,0}\To{}(R^{*}_{i_{J}})_{i,0}$ with signs
associated to the
simplicial set of subsets of $S\setminus I$.

The graded space $R^{*}_{i_{J}}$ is already known for $J=S$ by
\ref{proofpitriv} : we have $R^{*}_{i_{S}}=\oFl[1-d+2i]$. Let us thus fix $J\subsetneq S$.
Using Frobenius reciprocity and (\ref{cohofacdir}) we then  have 
isomorphisms
$$ \ext{*+i}{\bar h_{i}}{i_{J}}{G/\varpi^{\ZM}}\simto
(R_{i_{J}}^{*})_{i,0} $$
for $i\in\{0, \cdots, d-1\}$.
Further, using equality (\ref{hi}),
Proposition \ref{calculext2}, and Corollary \ref{calculext1} ii), we
get isomorphisms
\begin{eqnarray*}
(R_{i_{J}}^{*})_{i,0} &  \simeq  & \ext{*+i}{\bar
  h_{i}}{i_{J}}{G/\varpi^{\ZM}} 
 \simeq  
\ext{*+i}{ \pi_{\{1,\cdots, i\}}}{i_{J}}{G/\varpi^{\ZM}} 
 \simeq 
\ext{*+i}{ v_{\{1,\cdots, i\}}}{i_{J}}{G/\varpi^{\ZM}} \\
& \simeq & 
\left\{  
  \begin{array}{ll}
\ext{*+i-|S\setminus \{1,\cdots,i\}|}{ \oFl}{i_{J}}{G/\varpi^{\ZM}} 
 \simeq \bigwedge^{*+2i+1-d} Y_{J} 
& \hbox{if } \{i+1,\cdots, d-1\}\subseteq J \\
0 & \hbox{otherwise}
  \end{array}
\right..
\end{eqnarray*}
Observe that the smallest $J$ which contributes is $J(i,I):=I\cup
\{i+1,\cdots, d-1\}$.
In particular, the $E_{1}$ page of the spectral sequence is supported
in the vertical strip defined by 
$$-|S\setminus I|\leq p \leq -|J(i,I)\setminus I|.$$ 
Moreover, since $\dim Y_{J}= d-1-|J|$, we see that for each $p$ in the
above range, the column $E_{1}^{p*}$ is supported in the
range  
$$ d-1-2i\leq q \leq 2d-2-2i-p-|I|.$$
In other words, the $E_{1}$ page is supported in the half square with
left corner 
$$(-|S\setminus I|,d-1-2i)$$ and
right corners 
$$(-|J(i,I)\setminus I|, d-1-2i) \,\hbox{ and }\, (-|J(i,I)\setminus I|, 2d-2-2i-|J(i,I)|).$$

Now, we observe that the $E_{1}^{*\bullet}$ of our spectral
sequence is the same, up to some shifts, as that which occurs in 
\cite[Proof of Thm 1]{Orlikext} and \cite[Ch. X, Prop 4.7]{BW}.
We still have to compare the differential $d_{1}$ with that of these two references.
Using Corollary \ref{calculext2} ii) again, we see that for
$J'\supseteq J$ with $J'$ a \emph{strict subset} of $S$, 
 the non-zero map $(R^{*}_{i_{J'}})_{i,0}\To{}(R^{*}_{i_{J}})_{i,0}$
 is induced by the natural map $Y_{J'}\To{} Y_{J}$. It follows that
 for $p>-|S/I|$ (\emph{i.e.} everywhere except maybe on the first
 non-zero column),
 the differential $d_{1}^{pq}$ is the same as that of the two references
 cited above. In fact the only possible difference concerns
 $d_{1}^{-|S/I|, d-1-2i }$
for which we have no control yet in our setting, except in the trivial
case where $i=0$, because in this case, the $E_{1}$ is supported on
one point.
For $i>0$, in order to ensure that $d_{1}^{-|S/I|, d-1-2i }$
 is the same as in the two references cited
above, we have to prove that each map
\begin{equation}
 (R_{i_{S}}^{d-1-2i})_{i,0}\simeq \oFl \To{}
(R_{i_{J}}^{d-1-2i})_{i,0}\simeq \oFl\label{isom}
\end{equation}
is an isomorphism. However, since we know that all the other maps 
$ (R_{i_{J}}^{d-1-2i})_{i,0} \To{}
(R_{i_{J'}}^{d-1-2i})_{i,0}$ for $J'\subset J\subsetneq S$
are isomorphisms, it is sufficient to prove that (\ref{isom}) is an
isomorphism \emph{for a single $J$} !
For this, we look at the special case $I=\emptyset$, so that
the left corner is $(1-d, d-1-2i)$. If all the maps (\ref{isom}) were
zero, we would have  $E_{2}^{1-d,d-1-2i}\simeq \oFl$, creating a
non-zero $R^{-2i}_{\pi_{\emptyset}}$, which is absurd since
$R^{*}_{\pi_{\emptyset}}$ has to vanish for $*<-i$, and $i$ was
supposed to be positive.

Therefore, we have identified the first page of our spectral sequence
with that of \cite[Proof of Thm 1]{Orlikext} and \cite[Ch. X, Prop
4.7]{BW}, up to schifts.
Using their results, 
we get that $E_{2}^{pq}$ is always $0$ except in
the upper right corner of the triangle, where it is $1$-dimensional.
Therefore we get
$$ (R_{\pi_{I}}^{*})_{i,0} \simeq \oFl[-(2d -2 -2i -2|J(i,I)|+|I|)].$$
We still need to compute the schift $a=-(2d -2 -2i -2|J(i,I)|+|I|)$. Observe first that 
$2d-2-2|J(i,I)|= 2|\{1,\cdots, i\}\setminus I|$. Then, using
$i-|\{1,\cdots, i\}\setminus I|= |\{1,\cdots, i\}\cap I|$, we get
$a = 2|\{1,\cdots, i\}\cap I|-|I|$.
Eventually, using the equality $i+|I|=|\{1,\cdots, i\}\cap
I|+|\{1,\cdots, i\}\cup I|$, we get
$$a=|\{1,\cdots, i\}\cap I|-|\{1,\cdots, i\}\cup
I|+i=\partial_{I}(i).$$
The proof of Theorem \ref{theocomputRpi} is now complete.
However, it will be important to keep some track of the isomorphism
$ (R_{\pi_{I}}^{-\partial_{I}(i)})_{i,0}\simeq \oFl$ we have just obtained,
 when we study the Lefschetz operator in next subsection.
We may decompose this isomorphism in four steps :
\begin{enumerate}
\item The spectral sequence provides the isomorphism
$  (R_{\pi_{I}}^{-\partial_{I}(i)})_{i,0} \simeq
  (R_{i_{J(i,I)}}^{-\partial_{I}(i) +|J(i,I)\setminus I|})_{i,0}$.
\item Corollary \ref{calculext1} and Remark \ref{remext} exhibit an isomorphism
$$  (R_{i_{J(i,I)}}^{d-1-2i})_{i,0} \otimes \bigwedge^{\rm
 max}Y_{J(i,I)} \simto (R_{i_{J(i,I)}}^{-\partial_{I}(i) +|J(i,I)\setminus I|})_{i,0}.$$
\item The inclusion $\oFl=i_{S}\injo i_{J(i,I)}$ induces the isomorphism
$(R_{i_{S}}^{d-1-2i})_{i,0} \simeq (R_{i_{J(i,I)}}^{d-1-2i})_{i,0}$,
as was shown in the above proof.
\item The geometric input from \ref{proofpitriv} provides the 
  isomorphism  $(R_{\oFl}^{d-1-2i})_{i,0}\simeq \oFl$.
\end{enumerate}

\subsection{The Lefschetz operator}

We now study the Lefschetz operator recalled in the Introduction.
We refer the reader to \cite[2.2.4]{lefschetz} for the precise definition of this
operator and will contend ourselves with recalling the relevant details
when necessary in the proof of Theorem \ref{theoLpi} below.

\ali Our aim is to describe the operator
$L_{\pi}^{*}:\,R_{\pi}^{*}\To{}R_{\pi}^{*}[2](1)$ for $\pi$ a
unipotent elliptic representation. Since this operator is
$D^{\times}\times W_{K}$-equivariant, it decomposes as a sum 
$ L_{\pi}^{*}=\sum_{i,j} (L_{\pi}^{*})_{i,j}$ with 
$$ (L_{\pi}^{*})_{i,j}:\, (R_{\pi}^{*})_{i,j} \To{}
(R_{\pi}^{*+2})_{i-1,j}.$$
It also satisfies the following compatibility with torsion :
\begin{equation}
  \label{twist3}
  (L_{\nu_{_{G}}\pi}^{*})_{i,j} = (L_{\pi}^{*})_{i-1,j-1},
\end{equation}
where equality merely  means that these morphisms are part of a
commutative diagram involving isomorphisms (\ref{twist2}).
Thanks to (\ref{twist3}), we may restrict our attention to the case $j=0$.

Now, recall from Theorem \ref{theocomputRpi} that each
$(R_{\pi_{I}}^{*})_{i,0}$ is zero unless $I\subseteq S$. In the latter
case, it is $1$-dimensional and concentrated in
degree $-\partial_{I}(i)$. Therefore $(L_{\pi_{I}}^{*})_{i,0}$ is
necessarily $0$ as soon as $\partial_{I}(i) \neq \partial_{I}(i-1)+2$,
which by Fact \ref{defpartial} is equivalent to $i\notin I$.
 The following theorem asserts that $(L_{\pi_{I}}^{*})_{i,j}$ is
non-zero in the remaining cases.

\begin{theo}
  \label{theoLpi}
Let $I$ be a subset of $S$ and let $i\in I$.
Then the operator
$$ (L_{\pi_{I}}^{*})_{i,0}:\,
(R_{\pi_{I}}^{*})_{i,0}\simeq\oFl[\partial_{I}(i)]\To{}
 (R_{\pi_{I}}^{*})_{i-1,0}[2]\simeq \oFl[\partial_{I}(i-1)+2]$$
is non-zero, and thus is an isomorphism.
\end{theo}

\begin{proof}
As in the proof of Theorem \ref{theocomputRpi}, the crucial input
comes from geometry, which rules out the case of the trivial
representation $\oFl=\pi_{S}$.  Indeed, recall from \ref{proofpitriv}
that the period map provides us with isomorphisms
$$ R_{\pi_{S}}^{*}[1-d] \simeq  \left(H^{*}(\PM^{d-1,\rm ca},\oFl)\right)^{\vee}
= \bigoplus_{i=0}^{d-1}\oFl[2i](i).$$
But by its mere definition, the Lefschetz operator of
\cite[2.2.4]{lefschetz} induces the tautological Lefschetz operator on
$\PM^{d-1}$, namely that given by the Chern class of the tautological
sheaf. It is well known to induce isomorphisms 
$H^{i}(\PM^{d-1,\rm ca},\oFl) \simto H^{i+2}(\PM^{d-1,\rm
  ca},\oFl)(1)$
for $0\leq i<2d-2$, thereby proving the theorem for $I=S$.

We now consider a general $I\subset S$. We will use the four steps
 gathered in the end of Paragraph \ref{nonvanishing}, and
which summarize the origin of the isomorphism
$(R_{\pi_{I}}^{-\partial_{I}(i)})_{i,0}\simeq \oFl$. Motivated by step
iii) in that list, we consider for any $J\subset S$ the following
commutative diagram, which is functorially induced by the inclusion map
$i_{S}\injo i_{J}$
$$ 
\xymatrix{ (R_{i_{S}}^{d-1-2i})_{i,0} \ar[r] \ar[d]_{L_{i_{S}}^{d-1-2i}} &
(R_{i_{J}}^{d-1-2i})_{i,0} \ar[d]^{L_{i_{J}}^{d-1-2i}} \\
(R_{i_{S}}^{d-1-2i+2})_{i-1,0} \ar[r]  &
(R_{i_{J}}^{d-1-2i+2})_{i-1,0}
}
$$
The two horizontal maps were shown to be isomorphisms in
\ref{nonvanishing}, and the left vertical map has just been shown to be
so. We conclude that $(L_{i_{J}}^{d-1-2i})_{i,0}$ is an isomorphism.

Further, let us consider the diagram for $k\in\NM$
$$\xymatrix{
(R_{i_{J}}^{d-1-2i})_{i,0}\otimes_{\oFl} \bigwedge^{k}Y_{J}
\ar[r] \ar[d]_{L_{i_{J}}^{d-1-2i}\otimes\id} &
(R_{i_{J}}^{d-1-2i+k})_{i,0} \ar[d]^{L_{i_{J}}^{d-1-2i+k}} \\
(R_{i_{J}}^{d-1-2i+2})_{i-1,0}\otimes_{\oFl} \bigwedge^{k}Y_{J}
\ar[r]  &
(R_{i_{J}}^{d-1-2i+2+k})_{i-1,0} 
}$$
The horizontal maps are explained in Remark \ref{gradedmodule} and the
functoriality of these maps insures that the diagram is commutative.
It follows from the identification 
$(R_{i_{J}}^{*})_{i,0}\simeq \ext{*+i}{v_{\{1,\cdots,i\}}}{i_{J}}{G/\varpi^{\ZM}} $
 explained in the course of \ref{nonvanishing}, together with 
Remark \ref{remext} that these maps are isomorphisms.
Since the left vertical map has just been shown to be an isomorphism, so is 
the right one $(L_{i_{J}}^{d-1-2i+k})_{i,0}$.

Recall now the notation $J(i,I)=I\cup\{i+1,\cdots,d-1\}$ of
\ref{nonvanishing}, and observe that $J(i-1,I)=J(i,I)$ since we assume
$i\in I$. Recall also that
$\partial_{I}(i)=\partial_{I}(i-1)+2$ under this assumption, and
consider the diagram
$$\xymatrix{
(R_{\pi_{I}}^{-\partial_{I}(i)})_{i,0}
\ar[r] \ar[d]_{L_{\pi_{I}}^{-\partial_{I}(i)}} &
(R_{i_{J(i,I)}}^{-\partial_{I}(i)+|J(i,I)\setminus I|})_{i,0} 
\ar[d]^{L_{i_{J(i,I)}}^{-\partial_{I}(i)+|J(i,I)\setminus I|}} \\
(R_{\pi_{I}}^{-\partial_{I}(i-1)})_{i-1,0}
\ar[r]  &
(R_{i_{J(i,I)}}^{-\partial_{I}(i-1)+|J(i,I)\setminus I|})_{i-1,0} 
}$$
where the horizontal maps are provided by the spectral sequence
considered in \ref{nonvanishing} (these are edge maps once we know
enough on the support of the spectral sequence).
These maps were shown to be isomorphisms in \ref{nonvanishing}, and we
have just proved that the vertical right hand map is also an
isomorphism. We conclude that $L_{\pi_{I}}^{-\partial_{I}(i)}$ is an
isomorphism, as desired.

\end{proof}

\alin{Recollection and proof of the Main Theorem}
We now prove the theorem announced in the Introduction.
In particular we forget all gradings.
We first assume  that $\pi$ is a unipotent (or principal series) elliptic representation.
Let $I$ be the strict subset of $\wt{S}$ such that $\pi\simeq \pi_{I}$.
By (\ref{decompsemsimpl}) and \ref{theocomputRpi}, we have
$$ R_{\pi}^{*,\rm ss} \simeq 
\bigoplus_{i,j=0}^{d-1} (R_{\pi_{I}}^{*})_{i,j} \simeq
\bigoplus_{j\notin I} \nu_{_{D}}^{j}\otimes 
(R_{\pi}^{*,\rm ss})_{j}  
\,\hbox{ with }\, (R_{\pi}^{*,\rm ss})_{j}:= \bigoplus_{i=0}^{d-1}
(R_{\pi_{I}}^{*})_{i,j} = \bigoplus_{i=0}^{d-1} \nu_{_{W}}^{i}.$$ 
According to Theorem \ref{theoLpi} and the explicit description of
Proposition \ref{WDexplicit} we have 
$$ ((R_{\pi}^{*,\rm ss})_{0}, L_{\pi}^{*}) \simeq (\sigma^{\rm ss}(\pi), L(\pi)).$$
Applying again Theorem \ref{theoLpi} to $c^{-j}I$ and using
compatibility with twisting (\ref{twist3}), we get for any $j\notin I$
$$ ((R_{\pi}^{*,\rm ss})_{j}, L_{\pi}^{*}) \simeq 
\nu_{W}^{j}\otimes (\sigma^{\rm ss}(\pi_{c^{-j}I}), L(\pi_{c^{-j}I}))
\simeq
(\sigma^{\rm ss}(\pi), L(\pi)).$$
Recalling now Proposition \ref{LJexplicite}, we eventually get 
$$ (R_{\pi}^{*}, L_{\pi}^{*})^{\rm ss} \simeq |\LJ(\pi)|\otimes(\sigma^{\rm ss}(\pi), L(\pi)),$$
as desired. 

In order to finish the proof of the Main Theorem, we still
have to deal with the case when $\pi$ is not elliptic. In this case we
must show that $R_{\pi}^{*}=0$. Here we use the full force of the
Vign\'eras-Zelevinski classification in \cite[V.12]{Viginduced}\footnote{A
  more detailed account of this classification will appear in a current
  work by Minguez and Secherre}. 
Following this classification, there is a proper parabolically
induced representation $\iota$ which contains $\pi$ as a subquotient
with multiplicity one, and all other subquotients $\pi'$ of which
satisfy the condition $\lambda_{\pi'}<\lambda_{\pi}$. Here,
$\lambda_{\pi}$ is the partition associated to $\pi$ via the
successive highest derivatives.
Hence, arguing by induction on $\lambda_{\pi}$, we see that it
suffices to prove that $R_{\iota}^{*}=0$.
Write $\iota=i_{P}(\tau)$ for some proper standard parabolic subgroup
$P=MU$ and some irreducible representation $\tau$. Then
$R_{\iota}^{*}=\bigoplus_{i=0}^{d-1}\ext{*}{r_{P}v_{\{1,\cdots,
    i\}}}{\tau}{M}$.
But since $\pi$ is not elliptic, the cuspidal support of $\tau$ 
is disjoint from $W.\delta$.
Therefore, item i) of Lemma \ref{lemext} shows that each Ext group
occuring in the above sum vanishes.

\alin{Remark on non-unipotent representations}
The Main Theorem may remain true for any irreducible
$\oFl$-representation $\pi$ of $G$, under the Coxeter congruence
hypothesis. In fact, much is already known ; Boyer has described the
cohomology of the whole tower and has announced recently that the
integral cohomology is torsion-less. This allows to split the full complex
according to weights.
Then our arguments, which are somehow inductive on the
``Whittaker level'', work fine for arbitrary elliptic representations,
except that the induction has to be initialized at some point. For
unipotent representations, the initialization was the computation of
$(R_{\oFl}^{*},L_{\pi}^{*})$ thanks to the period map. 

All in all, our arguments
show that the Main Theorem is true for \emph{any representation}
$\pi$, provided it holds true for \emph{any super-Speh
  representation}, in the sense of \cite[Def 2.2.3]{jlmodl}.

\end{document}